\documentclass{amsart}
\usepackage{graphicx}

\usepackage{amsfonts}
\usepackage{amssymb}

\usepackage{enumitem}

\usepackage{amsmath}
\usepackage{latexsym}
\usepackage{epsfig}
\usepackage{xspace}
\usepackage{xcolor}
\usepackage[mathscr]{eucal}
\usepackage{color}
\usepackage{cite}
\usepackage{hyperref}
\usepackage[capitalize]{cleveref}

\newtheorem{theorem}{Theorem}
\newtheorem{lemma}[theorem]{Lemma}

\newtheorem{definition}[theorem]{Definition}
\newtheorem{example}[theorem]{Example}

\newtheorem{corollary}[theorem]{Corollary}
\newtheorem{proposition}[theorem]{Proposition}
\newtheorem{assumption}[theorem]{Assumption}
\newtheorem{remark}[theorem]{Remark}

\newcommand{\hmin}{h_{\min}}
\newcommand{\hmax}{h_{\max}}

\newcommand{\expect}[1]{\mathbb{E}\left[  {#1} \right]}

\newcommand{\normEc}[1]{\mathcal{E}_c^2(#1)}

\newcommand{\Dt}{\Delta t}
\newcommand{\Dx}{{\Delta x}}
\newcommand{\DW}{\triangle W}

\newcommand{\vecu}{\vec{u}}

\newcommand\bc{\begin{center}}
\newcommand\ec{\end{center}}
\newcommand\bi{\begin{itemize}}
\newcommand\ei{\end{itemize}}
\newcommand\be{\begin{enumerate}}
\newcommand\ee{\end{enumerate}}
\newcommand\bd{\begin{definition}}
\newcommand\ed{\end{definition}}
\newcommand\bt{\begin{theorem}}
\newcommand\et{\end{theorem}}
\newcommand\bp{\begin{proposition}}
\newcommand\ep{\end{proposition}}
\newcommand\bcor{\begin{corollary}}
\newcommand\ecor{\end{corollary}}
\newcommand\bx{\begin{exercise}}
\newcommand\ex{\end{exercise}}
\newcommand\beg{\begin{example}}
\newcommand\eeg{\end{example}}
\newcommand\bl{\begin{lemma}}
\newcommand\el{\end{lemma}}
\newcommand\bea{\begin{eqnarray*}}
\newcommand\eea{\end{eqnarray*}}

\newcommand{\N}[1]{N^{(#1)}}
\newcommand{\GVII}{{\Gamma_{\varepsilon,m}}}

\newcommand{\noteB}[1]{{\color{black}{#1}}}

\begin{document}

\title[Adaptive methods for stochastic systems]{Adaptive Euler methods for stochastic systems with non-globally Lipschitz coefficients}

\author{C\'onall Kelly}
\email{conall.kelly@ucc.ie}
\address{School of Mathematical Sciences, University College Cork, Ireland.}

\author{Gabriel Lord}
\email{gabriel.lord@ru.nl}
\address{Department of Mathematics, IMAPP, Radboud University, Nijmegen, The Netherlands.}

\begin{abstract}
We present strongly convergent explicit and semi-implicit adaptive numerical schemes for systems of semi-linear stochastic differential equations (SDEs) where both the drift and diffusion are not globally Lipschitz continuous. Numerical instability may arise either from the stiffness of the linear operator or from the perturbation of the nonlinear drift under discretization, or both. Typical applications arise from the space discretisation of an SPDE, stochastic volatility models in finance, or certain ecological models.
Under conditions that include montonicity, we prove that a timestepping strategy which adapts the stepsize based on the drift alone is sufficient to control growth and to obtain strong convergence with polynomial order. The order of strong convergence of our scheme is $(1-\varepsilon)/2$, 
for $\varepsilon\in(0,1)$, where $\varepsilon$ becomes arbitrarily small as the number of finite moments available  for solutions of the SDE increases. Numerically, we compare the adaptive semi-implicit method to a fully drift-implicit method and to three other explicit methods. 
Our numerical results show that overall the adaptive semi-implicit method is robust, efficient, and well suited as a general purpose solver.
\end{abstract}



\maketitle

\section{Introduction}\label{Sec:Intro}
Consider the $d$-dimensional semi-linear stochastic differential equation (SDE) of It\^o type
\begin{equation}
  dX(t)=[AX(t)+f(X(t)]dt+g(X(t))dW(t),\quad t\in[0,T];
\quad X(0) \in \mathbb{R}^d, \label{eq:SDE}
\end{equation}
where $T>0$, $A\in\mathbb{R}^{d\times d}$, $f:\mathbb{R}^d\to\mathbb{R}^d$, $g:\mathbb{R}^d\to\mathbb{R}^{d\times m}$, and $W$ is an $m$-dimensional Wiener process.
We suppose that the drift coefficient $f$ and the diffusion coefficient $g$ together satisfy polynomial bounds and a monotone condition permitting $g$ to grow superlinearly as long as that growth is countered sufficiently strongly by $f$. Global Lipschitz bounds are not required. For example, consider $f(x)=-x^2$ with $g(x)=x^{3/2}$ or $f(x)=-x^5$ with $g(x)=x^2$. Such applications arise in finance: for example the Lewis stochastic volatility model \cite{Lewis2000} which has a polynomial diffusion coefficient of order $3/2$. It was shown in \cite{HJ2011} that the explicit Euler-Maruyama method with constant stepsize fails to converge in the strong sense to solutions of \eqref{eq:SDE} if either the drift or the diffusion coefficients grow superlinearly. Also, as noted in \cite{Burrages2002}, fixed stepsize schemes may need to use very small stepsizes when the SDE being solved is stiff. We address these issues here by a semi-implicit scheme with adaptive timestepping.

In \cite{KeLo2016} a class of timestepping strategies, referred to as admissible, was motivated for the numerical discretisation of SDEs where the drift satisfies a one-sided Lipschitz coefficient and the diffusion satisfies a global Lipschitz bound. An admissible strategy uses the present value of the numerical trajectory to select the next timestep to avoid spuriously large drift responses. This is distinct from the error control approach in (for example) \cite{GainesLyons1997,Burrages2002,IJE2015}.

Timesteps selected by an admissible strategy are subject to upper and
lower limits $h_{\text{max}}$ and $h_{\text{min}}$ in a fixed ratio
$\rho$, with $h_{\text{max}}$ serving as a convergence parameter and
$h_{\text{min}}$ serving to ensure that the simulation completes in a
reasonable time. If the strategy attempts to select a timestep smaller
than $h_{\text{min}}$, then a backstop method is applied instead over
a single step of length $h_{\text{min}}$. It was proved in
\cite{KeLo2016} that the explicit Euler-Maruyama method over a random
mesh generated by an admissible timestepping strategy is strongly
convergent in $h_{\text{max}}$ with order $1/2$. The proof relied upon
$p^{th}$-moment bounds on the supremum of solutions of the underlying
SDE. Note also the adaptive
approach in \cite{Giles2017} which is consistent with the
admissibility condition of \cite{KeLo2016}.

Here, we examine more general SDEs and consider simultaneously both explicit and semi-implicit Euler-Maruyama schemes. Due to the monotone condition on the drift and diffusion terms, our analysis must contend with only a finite number of available bounded SDE moments (see for example the estimates provided by parts (i) and (ii) of Lemma \ref{LPSbook}). Unlike in \cite{KeLo2016}, we characterize precisely the backstop scheme and integrate it into the analysis in a way that is compatible with taking a random number of timesteps. 
In this way we show that a class of admissible timestepping strategies depending only on the drift response can be used to ensure that both the explicit and semi-implicit adaptive Euler-Maruyama schemes are strongly convergent to solutions of \eqref{eq:SDE} with order $(1-\varepsilon)/2$, in the sense that for any $\varepsilon\in(0,1)$, there exists $C_{\varepsilon}>0$, independent of $h_{\text{max}}$ such that 
\begin{equation*}
\mathbb{E}\left[\|X(T)-\widetilde Y_N\|^2\right]\leq C_{\varepsilon}h_{\text{max}}^{1-\varepsilon},
\end{equation*} 
where $\widetilde Y_N$ is value of the numerical scheme at time $T$, and $\|\cdot\|$ is the $l^2$ norm. The reduction in the order of strong convergence in our main result
(when compared to that in \cite{KeLo2016}) is a direct consequence of
the loss of global Lipschitz continuity in the diffusion
coefficient. If we reimpose global Lipschitz continuity on the
diffusion, we recover a strong convergence order of $1/2$, and if we
decompose the drift of \eqref{eq:SDE} so that $A=0$, we recover the
main result of \cite{KeLo2016}: see Remark \ref{rem:noStrongMom} for
more details of this.

The nature of the monotone condition is such that a timestepping scheme
which is admissible, and which can therefore successfully control the drift
response, will also be sufficient to control the diffusion response. 
It is well documented that 
the structure of the drift function (both linear and nonlinear) under
discretisation may have local dynamics that render the stability of
equilibria vulnerable to the effects of perturbation, either
stochastic or numerical \cite{AKMR,BK2012,HMS2002,HT1993,HJ2011}. 

Our method handles stiffness leading to potential instability in the discretisation in two distinct ways.  Where there is a classic (deterministic) stiff linear system we are able to treat this term implicitly without sacrificing numerical efficiency. Adaptive timestepping then treats nonlinear stability under stochastic perturbation. Thus, we deal with each source of potential instability separately, as would a stochastic IMEX-type method.  The use of an implicit approach to deal with the linear part of the drift avoids any consideration of potential interactions between it and the diffusion or between it and the nonlinear part of the drift. Note that the decomposition of the drift into the form $AX(t)+f(X(t))$ is determined by the modeller, and when $A=0$ the convergence analysis in this article applies equally to a fully explicit method if desired.

The literature already contains numerical schemes with fixed stepsizes that converge
strongly to solutions of SDEs with coefficients that satisfy local
Lipschitz and monotone conditions. Several of these extend the idea of
taming as introduced in \cite{HJK2012}, which rescales the functional
response of the drift coefficient in the scheme; they do so by
allowing the entire stochastic Euler map to be rescaled by some
combination of drift and diffusion responses. For example, see the
balanced method introduced in~\cite{TretyakovZhang}
and the variant presented in~\cite{Sabinas2013arxiv}, which
are both strongly convergent in this setting. 
The projected Euler method of \cite{beyn16} handles runaway trajectories by projecting them back onto a ball of radius inversely proportional to the step size; hence the authors control moments of the numerical solution.
It was shown in~\cite{MaoSzpruch2012} that a
drift-implicit discretisation could also ensure strong convergence in
our setting. \noteB{Finally we highlight \cite{HutzenthalerJentzen2}, which treats SDEs and SPDEs with non-globally monotone coefficients.}

In Section \ref{sec:num}, we compare the numerical performance of a selection of these methods
to that of the adaptive scheme presented in this article. We note this selection cannot be exhaustive and there are a growing number of variations; 
see for example \cite{HutzenthalerJentzen, Mao15,Mao16,Sabinas2016,SzpruchZhang,ZhangMa}. However our examples illustrate 
some of the drawbacks of fixed-step explicit schemes (when
linear stability is an issue) and where for fixed, relatively large $h$, the taming 
perturbation which imposes convergence may change the dynamics of the
solution. 
Compared to the fixed-step explicit methods, our numerical results show that the semi-implicit adaptive method gives consistently reliable numerical convergence results.
It is also more efficient than the drift-implicit scheme for SODEs, though this comparison is less clear for the SPDE example.

The structure of the article is as follows. In Section \ref{sec:setting} we describe the monotone condition and polynomial bounds that must be satisfied by $f$ and $g$, and provide the $p^{th}$-moment bounds satisfied by the solutions of \eqref{eq:SDE} within that framework. 
In Section \ref{sec:timestep}, we introduce the semi-implicit Euler-Maruyama method that, applied stepwise over a random mesh and combined with an appropriate backstop method, is the focus of the article. A mathematical definition for meshes produced by admissible timestepping strategies is provided, and conditional moment bounds for the SDE solution associated with these meshes are derived. 
In Section \ref{sec:main}, we present our main convergence result and state several technical lemmas, with proofs provided in Section \ref{app:technical}. 
In Section \ref{sec:num}, we carry out a comparative numerical investigation of strongly convergent schemes from the selection discussed above.

\section{Setting}\label{sec:setting}
Throughout the paper, $\mathbb{N}$ denotes the set $\{0,1,2,\ldots\}$, $\|\cdot\|$ denotes the $l^2$ norm of a
$d$-dimensional vector, $\|\cdot\|_F$ the Frobenius norm of a
$d\times m$-dimensional matrix, and for any $x\in\mathbb{R}^d$ and
$i=1,\ldots,m$, $g_i(x)$ denotes the $i^{\text{th}}$ column of the
diffusion coefficient matrix $g(x)$. For $a,b\in\mathbb{R}$ let $a
\vee b$ denote $\max\{a,b\}$. For any $A\in\mathbb{R}^{d\times d}$, we let $A^{1/2}\in\mathbb{C}^{d\times d}$ denote the matrix such that $(A^{1/2})^2=A$. Let $(\mathcal{F}_t)_{t\geq 0}$ be the natural filtration of $W$. To ensure the existence of a unique strong solution for \eqref{eq:SDE} \noteB{(in the sense of \cite[Definition 2.2.1]{Mao})} over the interval $[0,T]$, it suffices to place local Lipschitz and monotone conditions on $f$ and $g$:
\begin{assumption}\label{assum:locLipMon}
For each $R>1$ there exists $L_R>0$ such that
\begin{equation*}
\|f(x)-f(y)\|+\|g(x)-g(y)\|_F\leq L_R\|x-y\|,
\end{equation*}
for $x,y\in\mathbb{R}^d$ with $\|x\|\vee\|y\|\leq R$, and there exists $c\geq 0$ such that for some $p\geq 0$
\begin{equation}\label{eq:Mon}
\langle x-y,f(x)-f(y)\rangle+\frac{p+1}{2}\|g(x)-g(y)\|_F^2\leq c(\|x-y\|^2),\quad x,y\in\mathbb{R}^d.
\end{equation}
\end{assumption}
We also require a set of polynomial bounds on the derivatives of $f$ and $g$, and hence on $f$ and $g$ themselves. The minimum value of $p$ in \eqref{eq:Mon} required to prove our main strong convergence result depends on the order of these bounds.
\begin{assumption}\label{assum:polybounds}
Suppose $f:\mathbb{R}^d\to\mathbb{R}^d$ and $g:\mathbb{R}^d\to\mathbb{R}^{d\times m}$ are continuously differentiable with derivatives bounded as follows: for some $c_j,\gamma_0,\gamma_1\geq 0$; $j=1,\ldots,4$, 
we have
\begin{equation}\label{eq:fgDer}
\|Df(x)\|_F\leq c_1(1+\|x\|^{\gamma_0}),\quad \|Dg_i(x)\|_F\leq c_2(1+\|x\|^{\gamma_1}),\quad i=1,\ldots m,
\end{equation}
where $Df(x)\in\mathbb{R}^{d\times d}$ is the matrix of partial derivatives of $f$, and $Dg_i(x)\in\mathbb{R}^{d\times d}$ is the matrix of partial derivatives of the $i^{\text{th}}$ column of $g$, \noteB{and} 
\begin{equation}\label{eq:fgBound}
\|f(x)\|\leq c_3(1+\|x\|^{\gamma_0+1}),\quad \|g(x)\|_F\leq c_4(1+\|x\|^{\gamma_1+1}).
\end{equation}

\end{assumption}

We require that some of the moments of the solutions of \eqref{eq:SDE} are bounded over the interval $[0,T]$.
Further, \eqref{eq:Mon} in Assumption \ref{assum:locLipMon} implies (see, for example, Tretyakov \& Zhang~\cite{TretyakovZhang}) that there exists $c'\geq 0$ such that 
\begin{equation}\label{eq:pmonotone}
\langle x,f(x)\rangle+\frac{p-1}{2}\|g(x)\|_F^2\leq c'(1+\|x\|^2),\quad x\in\mathbb{R}^d.
\end{equation}
This is a special case of Khasminskii's condition using the Lyapunov-type function $V(x)=1+\|x\|^2$, and it guarantees the existence of a unique strong solution of \eqref{eq:SDE} over $[0,T]$ for any $T<\infty$ \noteB{(see \cite[Theorem 2.3.5]{Mao})}, while also ensuring $p^{th}$-moment bounds as follows:

\begin{lemma}\label{eq:momBoundExp}
Let $(X(t))_{t\in[0,T]}$ be the unique solution of \eqref{eq:SDE}. Suppose that \eqref{eq:pmonotone} holds for some $p\geq 2$ 
and \eqref{eq:fgBound} in Assumption \ref{assum:polybounds} holds, then there exists $M_{p,T}<\infty$ such that
\noteB{
\begin{equation}\label{eq:momStrong}
\mathbb{E}\left[\sup_{0\leq t\leq T}\|X(t)\|^{p-2\gamma_1}\right]\leq M_{p,T}.
\end{equation}
}

\end{lemma}
\begin{proof}
The proof of \eqref{eq:momStrong} follows from \cite[Lemma 4.2]{Mao16}, since the bound on $g$ provided by \eqref{eq:fgBound} implies Eq. (4.2) in that article, which we reproduce here as $$\|g(x)\|^2_F\leq K(1+\|x\|^r),\quad\text{for all}\quad x\in\mathbb{R}^d,$$ with $r=2\gamma_1+2$.
\end{proof}
To ensure sufficiently many bounded moments of the form \noteB{\eqref{eq:momStrong}} for our analysis to work, we now impose a lower bound on the value of $p$ in \eqref{eq:Mon} that depends on the order of the polynomial bounds on $f$ and $g$. This bound is associated with a tolerance parameter $\varepsilon$ which then appears in the the rate of strong convergence in Theorem \ref{thm:adaptConv}.
\begin{assumption}\label{assum:p}
Suppose that \eqref{eq:Mon} in Assumption \ref{assum:locLipMon} holds with
\noteB{\begin{equation}\label{eq:pLowerBound}
p\geq \max\left\{4\gamma_0,6\gamma_1\right\}+4+2^{q},
\end{equation}}
where $\gamma_0$ and $\gamma_1$ are as required in Assumption \ref{assum:polybounds}, and $\mathbb{N}\setminus\{0\}\ni q>1-\log_2\varepsilon$, where $\varepsilon\in(0,1)$ is a fixed tolerance parameter. 
\end{assumption}
Finally, note that the analysis in this article is also valid if the initial vector is random, $\mathcal{F}_0$-measurable, and $\mathbb{E}\|X(0)\|^p<\infty$.

\section{An adaptive semi-implicit Euler scheme with backstop}\label{sec:timestep}

The adaptive timestepping scheme under investigation in this article is based upon the semi-implicit Euler-Maruyama scheme over a random mesh $\{t_n\}_{n\in\mathbb{N}}$ on the interval $[0,T]$ given by
\begin{equation}\label{eq:Scheme}
 Y_{n+1}=Y_n+h_{n}AY_{n+1}+h_{n}f(Y_{n})+g(Y_n)\triangle W_{n+1},\quad n\in\mathbb{N},
\end{equation}
where $\triangle W_{n+1}:=W(t_{n+1})-W(t_n)$, and $\{h_n\}_{n\in\mathbb{N}}$ is a sequence of \noteB{positive} random timesteps and $\{t_n:=\sum_{i=1}^{n}h_{i-1}\}_{n\in\mathbb{N}\setminus\{0\}}$ with $t_0=0$. 
For the setting described in Section \ref{sec:setting}, we show that,
in order to ensure strong convergence with order $(1-\varepsilon)/2$
of the method \eqref{eq:Scheme} for any $\varepsilon\in(0,1)$, it is
sufficient to construct the stepsize sequence
$\{h_n\}_{n\in\mathbb{N}}$ in the same way as in \cite{KeLo2016}, demonstrating the applicability of this strategy to a significantly broader class of \noteB{SDEs}. 
We review the construction now.
\begin{definition}\label{def:filtration}
Suppose that each member of  $\{t_n:=\sum_{i=1}^{n}h_{i-1}\}_{n\in\mathbb{N}\setminus\{0\}}$, with $t_0=0$, is an $\mathcal{F}_t$-stopping time: i.e. $\{t_n\leq t\}\in\mathcal{F}_t$ for all $t\geq 0$, where $(\mathcal{F}_t)_{t\geq 0}$ is the natural filtration of $W$. \noteB{The filtration $(\mathcal{F}_t)_{t\geq 0}$ can be extended (see \cite{Mao}) to 
any $\mathcal{F}_t$-stopping time $\tau$ by
\[
\mathcal{F}_{\tau}:=\{B\in\mathcal{F}\,:\,B\cap\{\tau\leq t\}\in\mathcal{F}_t\}.
\]
In particular this allows us to condition on $\mathcal{F}_{t_n}$ at any point on the random time-set $\{t_n\}_{n\in\mathbb{N}}$.}
\end{definition}

\begin{remark}
Throughout the article, the index of a random sequence reflects its $\mathcal{F}_{t_n}$-measurability. For example, consider the timestep sequence $\{h_n\}_{n\in\mathbb{N}}$: each $h_n$ is $\mathcal{F}_{t_n}$-measurable. The only exception to this is $\{t_n\}_{n\in\mathbb{N}}$, since each $t_n$ is $\mathcal{F}_{t_{n-1}}$-measurable.
\end{remark}
\begin{assumption}\label{assum:h}
Suppose that each $h_n$ is $\mathcal{F}_{t_{n}}$-measurable, \noteB{and that there are constant minimum and maximum stepsizes $h_{\min}$ and $h_{\max}$ imposed in a fixed ratio $\rho$ so that  
\begin{equation}\label{eq:rho}
0<h_{\min}\leq h_n\leq h_{\max}<1,\quad h_{\text{max}}=\rho h_{\text{min}}.
\end{equation}}
\end{assumption}
\noteB{
\begin{definition}\label{def:N}
For each $t\in[0,T]$, define the random integer $\N{t}$ such that
\[
\N{t}:=\max\{n\in\mathbb{N}\setminus\{0\}\,:\,t_{n-1}<t\}.
\] 
Set $N:=\N{T}$ and $t_N:=T$, so that $T$ is always the last point on the mesh.
\end{definition}

We note that $\N{t}$ is a.e. the index of the right endpoint of the step that contains $t$, i.e. $t\in[t_{\N{t}-1},t_{\N{t}}]$ a.s. Both $t_{\N{t}}$ and $t_{\N{t}-1}$ are $\mathcal{F}_t$-stopping times, and $\N{t}$ is supported on the finite set $\{\N{t}_{\min},\ldots,\N{t}_{\max}\}$, where 
\begin{equation}\label{eq:Nminmax}
\N{t}_{\min}:=\lfloor t/h_{\max}\rfloor
\quad\text{and}\quad \N{t}_{\max}:=\lceil t/h_{\min}\rceil.
\end{equation}}
\begin{remark}\label{rem:conditionalMoments}
In \eqref{eq:Scheme}, note that each $\triangle W_{n+1}=W(t_{n+1})-W(t_n)$ is taken over a random step of length $h_{n}=h_{n}(Y_n)$ and which depends on $\{W(s),\,s\in[0,t_n]\}$ through $Y_n$. Therefore $\triangle W_{n+1}$ is a function of values of the Wiener process up to time $t_n$, is not independent of $\mathcal{F}_{t_{n}}$, and there is no reason to expect that $\triangle W_{n+1}\sim\mathcal{N}(0,h_{n} I_{d\times d})$, where $I_{d\times d}$ is the identity matrix. 
However since $t_{n+1}$ is a bounded
$\mathcal{F}_{t_{n}}$-stopping time then, by Doob's optional sampling theorem (see, for example, \cite{Shiryaev96}),
\begin{equation}\label{eq:WcondMoments}
\mathbb{E}\left[\triangle W_{n+1}\big|\mathcal{F}_{t_{n}}\right]=\mathbf{0},\quad
\mathbb{E}\left[\|\triangle W_{n+1}\|^2\big|\mathcal{F}_{t_{n}}\right]=h_{n},\quad a.s.
\end{equation}
\end{remark}
In our main analysis, we use the following lemma on the boundedness of the moments of solutions of \eqref{eq:SDE} conditioned at points on our adaptive mesh. The proof is a modification of that of \cite[Theorem 2.4.1]{Mao}.
\begin{lemma}\label{lem:condSDEMoments}
Let $(X(t))_{t\in[0,T]}$ be a unique solution of \eqref{eq:SDE}, and suppose that \eqref{eq:pmonotone} holds for some $p\geq 2$. Then there exist constants $\nu_1$ and $\nu_2$ such that
\begin{equation*}
\expect{\|X(t)\|^p|\mathcal{F}_{t_n}}\leq \nu_1+\nu_2\|X(t_n)\|^p,\quad t\geq t_n\quad a.s.
\end{equation*}
\end{lemma}
We are now in a position to define the scheme which is the subject of this article, and which combines a semi-implicit Euler scheme over an adaptive mesh, generated according to an admissible timestepping strategy, with a backstop method.
\begin{definition}\label{def:finalScheme}
Define the map $\theta:\mathbb{R}^d\times\mathbb{R}^d\times\mathbb{R}^{m}\times\mathbb{R}^{+}\to\mathbb{R}^d$ such that
\[
\theta(x,y,z,h):=x+hAy+hf(x)+g(x)z,
\]
so that, if $\{Y_n\}_{n\in\mathbb{N}}$ is defined by the semi-implicit scheme \eqref{eq:Scheme}, then $$Y_{n+1}=\theta(Y_n,Y_{n+1},\triangle W_{n+1},h_{n}),\quad n\in\mathbb{N}.$$ 
Then we define a \emph{semi-implicit Euler scheme with backstop} as the sequence $\{\widetilde Y_n\}_{n\in\mathbb{N}}$ by 
\begin{multline}\label{eq:finalScheme}
\widetilde Y_{n+1}= \theta(\widetilde Y_{n},\widetilde Y_{n+1},\triangle W_{n+1},h_{n})\cdot\mathcal{I}_{\{h_{\text{min}}<h_{n}\leq h_{\text{max}}\}}\\
+\varphi(\widetilde Y_{n},\widetilde Y_{n+1},\triangle W_{n+1},h_{\text{min}})\cdot\mathcal{I}_{\{h_{n}= h_{\text{min}}\}},
\end{multline}
where $\{h_n\}_{n\in\mathbb{N}}$ satisfies the conditions of Assumption \ref{assum:h}. 
The map $\varphi:\mathbb{R}^d\times\mathbb{R}^d\times\mathbb{R}^{m}\times\mathbb{R}^+\to\mathbb{R}^d$ \noteB{characterises the form a user-selected backstop method. We require that}
\begin{multline}\label{eq:backstopError2}
\expect{\left\|\varphi(\widetilde Y_{n},\widetilde Y_{n+1},\triangle W_{n+1},h_{\text{min}})-X(t_{n+1})\right\|^2\big|\mathcal{F}_{t_n}}-\|\widetilde Y_{n}-X(t_{n})\|^2\\
\leq C_1 h_{\text{min}}^{2-\varepsilon}+ C_2 h_{\text{min}}\left\|\widetilde Y_{n}-X(t_{n})\right\|^2,\quad n\in\mathbb{N},\quad a.s., 
\end{multline}
for some positive constants $C_1$ and $C_2$, and $\varepsilon\in(0,1)$ the fixed parameter from Assumption \ref{assum:p}.
\end{definition}
\begin{remark}
Note that $\varphi$ will satisfy \eqref{eq:backstopError2} if the
backstop method is subject to a mean-square
consistency requirement that bounds the propagation of discretisation error over a single step. 
In practice, rather than
checking \eqref{eq:backstopError2} directly, we use as our backstop a
method that is known to be strongly convergent of order $1/2$ in this
setting: for the numerical experiments in Section \ref{sec:num} we use the
balanced method introduced by Tretyakov \&
Zhang~\cite{TretyakovZhang}, which satisfies a similar local accuracy
bound (see \cite[Eq. (2.9)]{TretyakovZhang}) \noteB{and corresponds to the map
\begin{equation}\label{eq:varphiBalanced}
\varphi(x,y,z,h)=x+\frac{f(x)h+\sqrt{h}g(x)z}{1+h\|f(x)\|+\sqrt{h}\sum_{i=1}^{m}\|g_i(x)z_i\|}.
\end{equation}
Examples of $h_{\min}$, $h_{\max}$, and $\rho$ for the scheme \eqref{eq:finalScheme} with backstop characterised by \eqref{eq:varphiBalanced} are given in Section 5.}
\end{remark}

Finally, we define the admissible timestepping
strategy (see also \cite{KeLo2016}).
\begin{definition}\label{def:hAdmiss}
Let $\{\widetilde Y_n\}_{n\in\mathbb{N}}$ be a solution of \eqref{eq:finalScheme} where $f$ and $g$ satisfy the conditions of Assumptions \ref{assum:locLipMon}, \ref{assum:polybounds}, and \ref{assum:p}. We say that $\{h_n\}_{n\in\mathbb{N}}$ is an \emph{admissible timestepping strategy} for \eqref{eq:finalScheme} if Assumption \ref{assum:h} is satisfied and there \noteB{exist} real non-negative constants $R_1, R_2<\infty$, \noteB{independent of $h_{\max}$}, such that whenever $h_{\text{min}}<h_{n}\leq h_{\text{max}}$,
\begin{equation}\label{eq:normfgBounds}
\|f(\widetilde Y_n)\|^2\leq R_1+R_2\|\widetilde Y_n\|^2,\quad n=0,\ldots,N-1.
\end{equation}
\end{definition}
For example (see \cite{KeLo2016}) the timestepping rule given by
\begin{equation*}
h_{n}=\hmax \times
\min \left\{
\max \left\{
\frac{1}{\|f(\widetilde Y_n)\|},\frac{\|\widetilde Y_n\|}{\|f(\widetilde Y_n)\|},\rho
\right\},1 \right\}
\end{equation*}
is admissible for \eqref{eq:finalScheme}.  Choosing the larger of $1/\|f(\widetilde Y_n)\|$ and $\|\widetilde Y_n\|/\|f(\widetilde Y_n)\|$
helps maximize the stepsize while maintaining its admissibility.
The backstop is needed since it may not always be possible to  control $\widetilde Y_n$ via timestep so that \eqref{eq:normfgBounds} holds. See Section \ref{sec:concl} for a more detailed comment.

\section{Strong convergence of the adaptive scheme}\label{sec:main}
\subsection{Preliminary lemmas}
These lemmas provide a regularity bound in time and an estimate on remainder terms from Taylor expansions of $f$ and $g$. Proofs are given in Section \ref{app:technical}.
\begin{lemma}
\label{lem:Xregularity}
Let $(X(t))_{t\in[0,T]}$ be a solution of \eqref{eq:SDE} with coefficients $f$ and $g$ satisfying the conditions of Assumptions \ref{assum:locLipMon}, \ref{assum:polybounds}, and \ref{assum:p}, and suppose that the sequence of random times $\{t_n\}_{n\in\mathbb{N}}$ is as in Definition \ref{def:filtration} and satisfies the conditions of Assumption \ref{assum:h}. Then for all $n\in\mathbb{N}$ and $p\geq 2$, there exists an a.s. finite and $\mathcal{F}_{t_n}$-measurable random variable $\bar{L}_{p,n}$ so that 
\begin{equation}\label{eq:holdBound}
\expect{\|X(s)-X(t_n)\|^p\big|\mathcal{F}_{t_n}}\leq \bar{L}_{p,n}|s-t_n|^{p/2},\quad \noteB{s\in[t_n,t_{n+1}]},\quad a.s.
\end{equation}
\end{lemma}
\noteB{
Now consider the Taylor expansions of $f$ and $g_i$, $i=1,\ldots,m$:
\begin{eqnarray*}
f(X(s))&=&f(X(t_n))+R_f(s,t_n,X(t_n));\\
g_i(X(s))&=&g_i(X(t_n))+R_{g_i}(s,t_n,X(t_n)),
\end{eqnarray*}
where the remainders $R_f$ and $R_{g_i}$ are given in integral form by
\begin{equation*}
\label{eq:Rf}
R_{z}(s,t_n,X(t_n)):= \int_0^1 Dz(X(t_n) + \tau(X(s)-X(t_n)))(X(s)-X(t_n)) d\tau,
\end{equation*}
and $z$ can be taken to read either $f$ or $g_i$. 
Furthermore let 
\begin{eqnarray*}
R_A(s,t_n,X(t_n)):= A[X(t_{n+1})-X(s)].
\end{eqnarray*}
We now give $\mathcal{F}_{t_n}$-conditional mean-square bounds on the integrals of these remainder terms. }
\begin{lemma}
  \label{LPSbook}
Let $(X(t))_{t\in[0,T]}$ be a solution of \eqref{eq:SDE} with coefficients $f$ and $g$
satisfying the conditions of Assumptions \ref{assum:locLipMon}, \ref{assum:polybounds}, and \ref{assum:p}. Let $\{t_n\}_{n\in\mathbb{N}}$ be as in Definition \ref{def:filtration}, satisfying the conditions of Assumption \ref{assum:h}. 

Then for any $\varepsilon\in(0,1)$ there is an a.s. finite and $\mathcal{F}_{t_n}$-measurable random variable $\bar{\Lambda}_{\varepsilon,n}>0$, and a constant $\Lambda_{\varepsilon}<\infty$, the latter independent of $\{h_{n}\}_{n\in\mathbb{N}}$ and $h_{\text{max}}$, such that
\begin{align*}
(i) &\quad
\expect{\left\|\int_{t_n}^{t_{n+1}}R_{f}(s,t_n,X(t_n))ds\right\|^{2}\bigg|\mathcal{F}_{t_n}}
 \leq  \bar{\Lambda}_{\varepsilon,n}h_{n}^{3-\varepsilon},\quad a.s;\\
(ii) &\quad \expect{\left\|\int_{t_n}^{t_{n+1}}R_{g_i}(s,t_n,X(t_n))dW(s)\right\|^{2}\bigg|\mathcal{F}_{t_n}}
 \leq  \bar{\Lambda}_{\varepsilon,n}h_{n}^{2-\varepsilon},\quad a.s;\\
\noteB{ (iii)} &\quad \noteB{\expect{\left\|\int_{t_n}^{t_{n+1}}R_{A}(s,t_n,X(t_n))dW(s)\right\|^{2}\bigg|\mathcal{F}_{t_n}}
 \leq  \bar{\Lambda}_{\varepsilon,n}h_{n}^{3-\varepsilon},\quad a.s;}\\
(iv) &\quad  \mathbb{E}[\bar{\Lambda}_{\varepsilon,n}]\leq \Lambda_{\varepsilon}.
\end{align*}
\end{lemma}

\begin{remark}
The notational convention used in Part (\noteB{iv}) of Lemma \ref{LPSbook} is extended throughout the paper to $\mathcal{F}_{t_n}$-adapted sequences for which there exists a finite uniform upper bound on the expectation of each term.
\end{remark}

\subsection{Main results}
In this section, we provide a lemma giving a bound on the one-step error for the combined scheme, followed by our main theorem which uses discrete Gronwall inequalities to produce an order of strong convergence.
\begin{lemma}\label{lem:main}
Let $(X(t))_{t\in[0,T]}$ be a solution of
\eqref{eq:SDE} with drift and diffusion coefficients $f$ and $g$
satisfying the conditions of Assumptions \ref{assum:locLipMon}, \ref{assum:polybounds}, and \ref{assum:p}. Let
$\{\widetilde Y_n\}_{n\in\mathbb{N}}$ be a solution of \eqref{eq:finalScheme} with initial value $\widetilde Y_0=X_0$ and admissible timestepping strategy $\{h_n\}_{n\in\mathbb{N}}$ satisfying the conditions of Assumption \ref{assum:h} and Definition \ref{def:hAdmiss}. 

Define the error sequence $\{E_n\}_{n\in\mathbb{N}}$ by $E_{n}:=\widetilde Y_{n}-X(t_{n})$. Then there exist 
a.s. finite and 
$\mathcal{F}_{t_n}$-measurable random variables $\bar\Lambda_{\varepsilon,n}, \bar\Gamma_{2,n}, \bar\Gamma_{3,n}^{(m)}$ with finite expectations independent of $n$, denoted $\Lambda_\varepsilon$, $\Gamma_2$, and $\Gamma_3^{(m)}$ respectively, such that
 \begin{multline}\label{eq:preGronwalls2}
 \expect{\|E_{n+1}\|^2|\mathcal{F}_{t_n}}-\|E_n\|^2\leq h_{\text{max}}Q\expect{\|E_{n+1}\|^2|\mathcal{F}_{t_n}}
  +h_{\text{max}}\Gamma_1\|E_n\|^2\\+h_{\text{max}}^2\bar\Gamma_{2,n}+\bar\Gamma_{3,n}^{(m)}h_{\text{max}}^{2-\varepsilon}+126 \bar\Lambda_{\varepsilon,n} h_{\text{max}}^{3-\varepsilon},\quad n\in\mathbb{N},
 \end{multline}
where $\varepsilon\in(0,1)$ is the fixed parameter from Assumption \ref{assum:p} \noteB{and constants $Q,\Gamma_1$ are given by
\begin{eqnarray}\label{def:Q}
Q&:=&\|A^{1/2}\|_F^2+\frac{3}{2}\|A\|_F^2;\\
\Gamma_1&:=&(2c+20R_2+2)\vee C_2,\label{eq:barGamma1}
\end{eqnarray}
where $c,R_2,C_2$ satisfy \eqref{eq:Mon} in Assumption \ref{assum:locLipMon}, \eqref{eq:backstopError2} in Definition \ref{def:finalScheme}, and \eqref{eq:normfgBounds} in Definition \ref{def:hAdmiss} respectively.}
\end{lemma}
\begin{proof}
For $h_{n}$ selected at time $t_n$, for some $n\in\mathbb{N}$, by an
admissible timestepping strategy, there are two possible cases
(denoted (I) and (II)), first, $h_{\text{min}}<h_{n}\leq
h_{\text{max}}$ and second, $h_{n}= h_{\text{min}}$. We consider each in turn. 

(I) In this case we rely on the bound \eqref{eq:normfgBounds} provided by the admissibility of the timestepping scheme. 
When $h_{\text{min}}<h_{n}\leq h_{\text{max}}$, $\widetilde Y_{n+1}$ is derived from $\widetilde Y_n$ using \eqref{eq:Scheme}, and we have
\begin{eqnarray*}
E_{n+1}&:=&\widetilde Y_n-X(t_n)+\int_{t_n}^{t_{n+1}}A[\widetilde Y_{n+1}-X(s)]ds+\int_{t_n}^{t_{n+1}}[f(\widetilde Y_n)-f(X(s))]ds\\
&&\qquad+\sum_{i=1}^m\int_{t_n}^{t_{n+1}}[g_i(\widetilde Y_n)-g_i(X(s))]dW_i(s).
\end{eqnarray*}
Expand $f$ and $g$ as Taylor series around $X(t_n)$ over the interval of
integration, and write
\begin{eqnarray*}
A[\widetilde Y_{n+1}-X(s)]&=&A[\widetilde Y_{n+1}-X(t_{n+1})]+A[X(t_{n+1})-X(s)]\\
&:=&AE_{n+1}+R_A(s,t_n,X(t_n)).
\end{eqnarray*}
Therefore
\begin{multline*}
E_{n+1}=E_n+\int_{t_n}^{t_{n+1}}AE_{n+1}ds+\int_{t_n}^{t_{n+1}}[f(\widetilde Y_n)-f(X(t_n))]ds\\+\sum_{i=1}^m\int_{t_n}^{t_{n+1}}[g_i(\widetilde Y_n)-g_i(X(t_n))]dW_i(s)
 +\underbrace{\int_{t_n}^{t_{n+1}}R_A(s,t_n,X(t_n))ds}_{:=\tilde{R}^A_{n+1}}\\ +\underbrace{\int_{t_n}^{t_{n+1}}R_f(s,t_n,X(t_n))ds}_{:=\tilde{R}^f_{n+1}}+\sum_{i=1}^m\underbrace{\int_{t_n}^{t_{n+1}}R_{g_i}(s,t_n,X(t_n))dW_i(s)}_{:=\tilde{R}^{g_i}_{n+1}},
\end{multline*}
which is
\begin{multline*}
E_{n+1}=E_n+h_{n}AE_{n+1}+h_{n}[f(\widetilde Y_n)-f(X(t_n))]\\
+[g(\widetilde Y_n)-g(X(t_n))]\triangle W_{n+1}+\tilde{R}^A_{n+1}+\tilde{R}^f_{n+1}+\sum_{i=1}^m\tilde{R}^{g_i}_{n+1}.
\end{multline*}
Let $Q$ be as defined in \eqref{def:Q}.
Then, using that $h_{\text{max}}\leq 1$ and the
inequality $2\langle x,y\rangle\leq \|x\|^2+\|y\|^2$, we find
\begin{multline*}
\|E_{n+1}\|^2\leq \|E_n\|^2+h_{n}Q\|E_{n+1}\|^2\\
+\underbrace{2h_{n}\langle f(\widetilde Y_n)-f(X(t_n),E_n\rangle+5\|[g(\widetilde Y_n)-g(X(t_n))]\triangle W_{n+1}\|^2}_{:=A_{n+1}}\\
+\underbrace{5h_{n}^2\|f(\widetilde Y_n)-f(X(t_n))\|^2}_{:=B_n}\\
+\underbrace{2\langle E_n,\tilde{R}^A_{n+1}+\tilde{R}^f_{n+1}\rangle+7\left\|\tilde{R}^A_{n+1}+\tilde{R}^f_{n+1}+\sum_{i=1}^m\tilde{R}^{g_i}_{n+1}\right\|^2}_{:=C_{n+1}}\\
+2 \sum_{i=1}^m\langle E_n,\tilde R^{g_i}_{n+1}\rangle+4h_{n}\left\langle f(\widetilde Y_n)-f(X(t_n)),[g(\widetilde Y_n)-g(X(t_n))]\triangle W_{n+1}\right\rangle\\
+2\left\langle E_n,[g(\widetilde Y_n)-g(X(t_n))]\triangle W_{n+1}\right\rangle.
\end{multline*}

We develop bounds on $\expect{A_{n+1}|\mathcal{F}_{t_n}}$, $\expect{B_n|\mathcal{F}_{t_n}}$, $\expect{C_{n+1}|\mathcal{F}_{t_n}}$ in turn. The terms \noteB{after $C_{n+1}$ on the RHS of the inequality} have zero conditional expectation, by \eqref{eq:WcondMoments} in  Remark \ref{rem:conditionalMoments}, and the fact that $E_n$ and each $\tilde R_{g_i}$ are conditionally independent with respect to $\mathcal{F}_{t_n}$,  with the latter an It\^o integral with zero conditional expectation.

By \eqref{eq:Mon} in Assumption \ref{assum:locLipMon},
\begin{eqnarray*}
\noteB{\expect{A_{n+1}|\mathcal{F}_{t_n}}}&\leq& 2h_{n}\langle f(\widetilde Y_n)-f(X(t_n)),E_n\rangle+5h_{n}\noteB{\|g(\widetilde Y_n)-g(X(t_n))\|_F^2}\\
&\leq&2ch_{n}\|E_n\|^2,\quad a.s.
\end{eqnarray*}

By \eqref{eq:normfgBounds} in Definition \ref{def:hAdmiss}, and \eqref{eq:fgBound} in Assumption \ref{assum:polybounds} we have
\begin{eqnarray*}
\lefteqn{\expect{B_n|\mathcal{F}_{t_n}}=B_n}\\
&=&5h_{n}^2\|f(\widetilde Y_n)-f(X(t_n))\|^2\\
&\leq&10h_{n}^2(\|f(\widetilde Y_n)\|^2+\|f(X(t_n))\|^2)\\
&\leq&10h_{n}^2(R_1+2R_2(\|E_n\|^2+\|X(t_n)\|^2)+4c_3^2(1+\|X(t_n)\|^{2\gamma_0+2}))\\
&=&20h_{n}^2R_2\|E_n\|^2\\
&&\qquad+10h_{n}^2(R_1+2R_2\|X(t_n)\|^2+4c_3^2(1+\|X(t_n)\|^{2\gamma_0+2})),\quad a.s.
\end{eqnarray*}
Next, by Jensen's inequality and part (i) of Lemma \ref{LPSbook}, 
\begin{align*}
\expect{\langle E_n,\tilde{R}^f_{n+1}\rangle|\mathcal{F}_{t_n}}&\leq \|E_n\|\expect{\|\tilde{R}^f_{n+1}\||\mathcal{F}_{t_n}}\leq \|E_n\|\sqrt{\expect{\|\tilde{R}^f_{n+1}\|^2|\mathcal{F}_{t_n}}}\\
&\leq \|E_n\|\sqrt{\bar\Lambda_{\varepsilon,n}}h_{n}^{(3-\varepsilon)/2}\leq \frac{1}{2}h_{n}\|E_n\|^2+\frac{1}{2}\bar \Lambda_{\varepsilon,n} h_{n}^{2-\varepsilon},\,\, a.s.
\end{align*}
We also have from part (iii) of Lemma \ref{LPSbook} $\expect{\langle E_n,\tilde{R}^A_{n+1}\rangle|\mathcal{F}_{t_n}}\leq\frac{1}{2}h_{n}\|E_n\|^2+\frac{1}{2}\bar\Lambda_{\varepsilon,n} h_{n}^{2-\varepsilon}$ a.s.
Applying parts (i)-(iii) of Lemma \ref{LPSbook} then gives
\begin{eqnarray*}
\expect{C_{n+1}|\mathcal{F}_{t_n}}&=&\expect{2\langle E_n,\tilde{R}^A_{n+1}+\tilde{R}^f_{n+1}\rangle+7\left\|\tilde{R}^A_n+\tilde{R}^f_{n+1}+\sum_{i=1}^{m}
\tilde{R}^{g_i}_{n+1}\right\|^2\bigg|\mathcal{F}_{t_n}}\\
&\leq& 2h_{n}\|E_n\|^2+(2+63m^2) \bar\Lambda_{\varepsilon,n} h_{n}^{2-\varepsilon}+126\bar \Lambda_{\varepsilon,n} h_{n}^{3-\varepsilon},\quad a.s. 
\end{eqnarray*}

Therefore we have
\begin{eqnarray*}
\lefteqn{\expect{\|E_{n+1}\|^2|\mathcal{F}_{t_n}}-\|E_n\|^2}\nonumber\\
&\leq& h_{n}Q\expect{\|E_{n+1}\|^2|\mathcal{F}_{t_n}}+h_{n}\left(2c+20h_{\text{max}}R_2+2\right)\|E_n\|^2\\
&&+10h_{n}^2\left(R_1+2R_2\|X(t_n)\|^2+4c_3^2(1+\|X(t_n)\|^{2\gamma_0+2})\right)\\
&&+(2+63m^2)\bar \Lambda_{\varepsilon,n} h_{n}^{2-\varepsilon}+126\bar \Lambda_{\varepsilon,n}h_{n}^{3-\varepsilon},\quad a.s.
\end{eqnarray*}

(II) Suppose that $h_{n}= h_{\text{min}}$. Here $\widetilde Y_{n+1}$ is
generated from $\widetilde Y_n$ via an application of the backstop method
over a single step of length $h_{\text{min}}$. This corresponds to a
single application of the map $\varphi$ and therefore the relation
\eqref{eq:backstopError2} is satisfied a.s.  

To combine the two cases (I) and (II), define  
the a.s. finite and $\mathcal{F}_{t_n}$-measurable random variables  
\begin{eqnarray}
\label{eq:barGamma2}
\bar\Gamma_{2,n}&:=&10\left(R_1+2R_2\|X(t_n)\|^2+4c_3^2(1+\|X(t_n)\|^{2\gamma_0+2})\right),\\
\bar\Gamma_{3,n}^{(m)}&:=&(2+63m^2) \bar\Lambda_{\varepsilon,n}\vee C_1,
\end{eqnarray}
where $C_1$ and $C_2$ are as given in \eqref{eq:backstopError2}. 
Since $Q,\bar\Lambda_{\varepsilon,n}\geq 0$ (the latter in the a.s. sense), we have for any $h_{n}$ selected by an admissible adaptive timestepping strategy, 
\begin{multline}\label{eq:consistency}
\expect{\|E_{n+1}\|^2|\mathcal{F}_{t_n}}-\|E_n\|^2\leq h_{n}Q\expect{\|E_{n+1}\|^2|\mathcal{F}_{t_n}}
+h_{n} \Gamma_{1}\|E_n\|^2\\+h_{n}^2\bar\Gamma_{2,n}+\bar\Gamma_{3,n}^{(m)} h_{n}^{2-\varepsilon}+126 \bar\Lambda_{\varepsilon,n} h_{n}^{3-\varepsilon},\quad a.s.
\end{multline}
Note that since $(a\vee b)\leq a+b$ when $a,b\geq 0$, by \eqref{eq:barGamma1} we may apply \eqref{eq:momStrong} to \eqref{eq:barGamma2} to show that, under Assumption \ref{assum:p},
\[
\expect{\bar\Gamma_{2,n}}\leq 10\left(R_1+2R_2\noteB{M_{2+2\gamma_1,T}}+4c_3^2\left(1+\noteB{M_{2(\gamma_0+\gamma_1)+2,T}}\right)\right)=:\Gamma_2<\infty,
\]
where $\noteB{M_{2+2\gamma_1,T}}$ and $\noteB{M_{2(\gamma_0+\gamma_1)+2,T}}$ are finite constants satisfying \noteB{\eqref{eq:momStrong}} for $p=2,2\gamma_0+2$ respectively. It can be similarly shown \noteB{under Assumption \ref{assum:p} that there exist finite constants $\Gamma_{3}^{(m)},\Lambda_{\varepsilon}$ such that $\expect{\bar\Gamma_{3,n}^{(m)}}\leq \Gamma_3^{(m)}$ and $\expect{\bar\Lambda_{\varepsilon,n}}\leq \Lambda_{\varepsilon}$}.
\end{proof}

The bound \eqref{eq:consistency} characterises the propagation of error in mean-square over a single step of the combined semi-implicit Euler scheme with backstop \eqref{eq:finalScheme}, and holds regardless of whether or not the timestepping strategy requires an application of the semi-implicit scheme or the backstop scheme.
\noteB{
\begin{assumption}\label{assum:extrahbound}
Finally, $h_{\max}$ is chosen so that there exists a constant $\delta\in[0,1]$ that does not depend on $h_{\max}$, such that 
\begin{equation}\label{eq:hmaxbound}
h_{\max}<\min\left\{\frac{\delta}{2\rho(Q+\Gamma_1)},\frac{1-\delta}{Q}\right\},
\end{equation}
where $Q$ is defined in \eqref{def:Q} and $\Gamma_1$ is defined in \eqref{eq:barGamma1}. It follows that there exists $\gamma<\infty$ such that
\begin{equation}\label{eq:deltagamma}
\frac{1}{1-h_{\max} 2\rho(Q+\Gamma_1)/\delta}<\gamma.
\end{equation}
\end{assumption}
Although these conditions are required in the proof of Theorem \ref{thm:adaptConv} we have observed no practical implications in our numerical experiments.

\begin{definition}\label{def:Ec}
Define an a.s. continuous process $(\normEc{t})_{t\in[0,T]}$ pathwise as the a.e. linear interpolant of $\|E_n\|^2$ and $\|E_{n+1}\|^2$ on each interval $[t_n,t_{n+1}]$ for $n=0,\ldots,N-1$: 
\begin{equation}\label{eq:interpolant}
\normEc{s}:=\frac{t_{n+1}-s}{h_n}\|E_n\|^2+\frac{s-t_n}{h_n}\|E_{n+1}\|^2,\quad s\in[t_n,t_{n+1}],\quad a.s.
\end{equation} 
\end{definition}

The accumulation of error in mean-square for \eqref{eq:finalScheme}, and hence the order of strong convergence, can now be estimated.
}
\begin{theorem}\label{thm:adaptConv}
Let $(X(t))_{t\in[0,T]}$ be a solution of
\eqref{eq:SDE} with drift and diffusion coefficients $f$ and $g$
satisfying the conditions of Assumptions \ref{assum:locLipMon}, \ref{assum:polybounds}, and \ref{assum:p}. Let
$\{\widetilde Y_n\}_{n\in\mathbb{N}}$ be a solution of  \eqref{eq:finalScheme} in Definition \ref{def:finalScheme},
with initial value $\widetilde Y_0=X_0$ and admissible timestepping strategy $\{h_n\}_{n\in\mathbb{N}}$ satisfying the conditions of Definition \ref{def:hAdmiss} and Assumption \ref{assum:extrahbound}. Then, if $\varepsilon\in(0,1)$ is the fixed parameter from Assumption \ref{assum:p}, there exists \noteB{$C_{\varepsilon,m,\rho,T}>0$}, independent of $h_{\text{max}}$ such that 
\noteB{
\begin{equation*}
\max_{t\in[0,T]}\mathbb{E}\left[\normEc{t}\right]\leq C_{\varepsilon,m,\rho,T}h_{\text{max}}^{1-\varepsilon},
\end{equation*}
where $\normEc{t}$ is as defined in Definition \ref{def:Ec}, 
and in particular 
\[
\left(\expect{\|X(T)-\widetilde{Y}_N\|^2}\right)^{1/2}\leq C_{\varepsilon,m,\rho,T}^{1/2}\hmax^{(1-\varepsilon)/2},
\]
where $N$ is as given in Definition \ref{def:N}.}
\end{theorem}

\begin{remark}
By setting $A=0$ in \eqref{eq:SDE} we obtain strong convergence of identical order of a
backstopped fully explicit Euler-Maruyama adaptive method.
\end{remark}

\begin{proof}[Theorem \ref{thm:adaptConv}].
\noteB{Fix $t\in[0,T]$, and let $\N{t}$ be as in Definition \ref{def:N}}. Multiply both sides of \eqref{eq:consistency} by the indicator variable $\mathcal{I}_{\{\noteB{\N{t}}\geq n+1\}}$ to get
 \begin{multline*}
 \expect{\|E_{n+1}\|^2|\mathcal{F}_{t_n}}\mathcal{I}_{\{\noteB{\N{t}}\geq n+1\}}-\|E_n\|^2\mathcal{I}_{\{\noteB{\N{t}}\geq n+1\}}\\
 \leq  {h_{n}}Q\expect{\|E_{n+1}\|^2|\mathcal{F}_{t_n}}\mathcal{I}_{\{\noteB{\N{t}}\geq n+1\}}+{h_{n}}\Gamma_{1}\|E_n\|^2\mathcal{I}_{\{\noteB{\N{t}}\geq n+1\}}\\
 +{h_{n}}^2\bar\Gamma_{2,n}\mathcal{I}_{\{\noteB{\N{t}}\geq n+1\}}+\bar\Gamma_{3,n}^{(m)} {h_{n}}^{2-\varepsilon}\mathcal{I}_{\{\noteB{\N{t}}\geq n+1\}}\\+126 \bar\Lambda_{\varepsilon,n} {h_{n}}^{3-\varepsilon}\mathcal{I}_{\{\noteB{\N{t}}\geq n+1\}},\quad a.s.
 \end{multline*}
  Since $\noteB{t_{\N{t}}}$ is a $\mathcal{F}_{t}$-stopping time, the event $\{\noteB{\N{t}}\leq n\}\in\mathcal{F}_{t_n}$ and therefore the indicator variable $\mathcal{I}_{\{\noteB{\N{t}}\geq n+1\}}$ is $\mathcal{F}_{t_n}$-measurable. Thus we have
 \begin{multline*}
 \expect{\|E_{n+1}\|^2\mathcal{I}_{\{\noteB{\N{t}}\geq n+1\}}|\mathcal{F}_{t_n}}-\|E_n\|^2\mathcal{I}_{\{\noteB{\N{t}}\geq n+1\}}\\
 \leq  h_{n}Q\expect{\|E_{n+1}\|^2\mathcal{I}_{\{\noteB{\N{t}}\geq n+1\}}|\mathcal{F}_{t_n}}+h_{n} \Gamma_{1}\|E_n\|^2\mathcal{I}_{\{\noteB{\N{t}}\geq n+1\}}\\
 +{h_{n}}^2\bar\Gamma_{2,n}\mathcal{I}_{\{\noteB{\N{t}}\geq n+1\}}+\bar\Gamma_{3,n}^{(m)} {h_{n}}^{2-\varepsilon}\mathcal{I}_{\{\noteB{\N{t}}\geq n+1\}}\\+126 \bar\Lambda_{\varepsilon,n} {h_{n}}^{3-\varepsilon}\mathcal{I}_{\{\noteB{\N{t}}\geq n+1\}},\quad a.s.
 \end{multline*}
 Since $\{\noteB{\N{t}}\geq n+1\}\subset \{\noteB{\N{t}}\geq n\}$, we have $\mathcal{I}_{\{\noteB{\N{t}}\geq n+1\}}(\omega)\leq\mathcal{I}_{\{\noteB{\N{t}}\geq n\}}(\omega)$ for all $\omega\in\Omega$. Take expectations on both sides, and since ${h_{n}}\leq h_{\text{max}}$ we have
 \begin{multline}\label{eq:preGronwalls3}
 \expect{\|E_{n+1}\|^2\mathcal{I}_{\{\noteB{\N{t}}\geq n+1\}}}-\expect{\|E_n\|^2\mathcal{I}_{\{\noteB{\N{t}}\geq n+1\}}}\\
 \leq  h_{\text{max}}Q\expect{\|E_{n+1}\|^2\mathcal{I}_{\{\noteB{\N{t}}\geq n+1\}}}+h_{\text{max}} \Gamma_{1}\expect{\|E_n\|^2\mathcal{I}_{\{\noteB{\N{t}}\geq n\}}}\\
 +h_{\text{max}}^2\Gamma_{2}+\Gamma_{3}^{(m)} h_{\text{max}}^{2-\varepsilon}+126 \Lambda_{\varepsilon} h_{\text{max}}^{3-\varepsilon},\quad a.s.
 \end{multline}
 
Now sum both sides of \eqref{eq:preGronwalls3} over $n=0,\ldots,\noteB{\N{t}_{\max}}-1$, \noteB{where $\N{t}_{\max}$ is the deterministic index in 
\eqref{eq:Nminmax}, to get (using the bound $t\leq T$)}
\begin{multline}\label{eq:noCDGI1}
\expect{\|E_\noteB{\N{t}}\|^2}=\expect{\|E_\noteB{\N{t}}\|^2\mathcal{I}_{\noteB{\N{t}}\geq \noteB{\N{t}}}}\\
\leq h_{\max}Q\sum_{n=0}^{\noteB{\N{t}_{\max}}-1}\expect{\|E_{n+1}\|^2\mathcal{I}_{\{\noteB{\N{t}}\geq n+1\}}}\\+h_{\max}\Gamma_1\sum_{n=0}^{\noteB{\N{t}_{\max}}-1}\expect{\|E_{n}\|^2\mathcal{I}_{\{\noteB{\N{t}}\geq n\}}}\\
+h_{\max}\rho T\Gamma_2+h_{\max}^{1-\varepsilon}\rho T\Gamma_3^{(m)}+126\Lambda_{\varepsilon}h_{\max}^{1-\varepsilon}\rho T.
\end{multline}
Bringing the sum inside the expectation on the RHS of \eqref{eq:noCDGI1} yields
\begin{multline}\label{eq:noCDGI2}
\expect{\|E_\noteB{\N{t}}\|^2}\leq\expect{h_{\max}Q\sum_{n=0}^{\noteB{\N{t}}-1}\|E_{n+1}\|^2+h_{\max}\Gamma_1\sum_{n=0}^{\noteB{\N{t}}-1}\|E_n\|^2}\\
+h_{\max}\rho T\Gamma_2+h_{\max}^{1-\varepsilon}\rho T\Gamma_3^{(m)}+126\Lambda_{\varepsilon}h_{\max}^{1-\varepsilon}\rho T.
\end{multline}
By a change in index in the second sum on the RHS of \eqref{eq:noCDGI2} (and since $\|E_0\|^2=0$ a.e.) we can write
\begin{multline}\label{eq:noCDGI3}
\expect{\|E_\noteB{\N{t}}\|^2}\leq\expect{h_{\max}Q\|E_\noteB{\N{t}}\|^2}+\expect{h_{\max}(Q+\Gamma_1)\sum_{n=0}^{\noteB{\N{t}}-1}\|E_n\|^2}\\
+h_{\max}\rho T\Gamma_2+h_{\max}^{1-\varepsilon}\rho T\Gamma_3^{(m)}+126\Lambda_{\varepsilon}h_{\max}^{1-\varepsilon}\rho T.
\end{multline}
\noteB{
Subtracting from both sides the first term on the RHS of \eqref{eq:noCDGI3}, and dividing through by $1-h_{\max}Q>\delta$ (which holds by \eqref{eq:hmaxbound} in Assumption \ref{assum:h}), we get
\begin{equation}
\expect{\|E_{\N{t}}\|^2}
\leq\frac{1}{\delta}h_{\max}(Q+\Gamma_1)\expect{\sum_{n=0}^{\N{t}-1}\|E_n\|^2} + \hmax^{1-\varepsilon}\GVII,\label{eq:noCDGI4}
\end{equation}
where
$\GVII:=\frac{\rho T}{\delta}\left(\Gamma_2+\Gamma_3^{(m)}+126\Lambda_{\varepsilon}\right)$, using $\hmax^{\varepsilon}<1$ by Assumption \ref{assum:h}.

It follows from \eqref{eq:interpolant} in Definition \ref{def:Ec} that a.s., 
\[
(t_{n+1}-s)\|E_n\|^2\leq h_n \normEc{s},\quad s\in[t_n,t_{n+1}],
\] 
and therefore by integration 
\begin{equation}\label{eq:discToCont}
h_{\min}\|E_n\|^2\leq 2\int_{t_n}^{t_{n+1}}\normEc{s}ds,\quad a.s.
\end{equation}
The a.s. continuity of $(\normEc{t})_{s\in[0,T]}$ implies the continuity and therefore boundedness over $[0,T]$ of $\expect{\normEc{t}}$.
Combining \eqref{eq:noCDGI4} and \eqref{eq:discToCont}, and using that $\hmax=\rho\hmin$, we get
\begin{eqnarray}
\lefteqn{\expect{\|E_{\N{t}}\|^2}}\nonumber\\
&\leq& \frac{2\rho}{\delta}(Q+\Gamma_1)\expect{\int_{0}^{t}\normEc{s}ds + \int_{t}^{t_{\N{t}}}\normEc{s}ds}
+ \hmax^{1-\varepsilon}\GVII\nonumber\\
&\leq& \frac{2\rho}{\delta}(Q+\Gamma_1)\expect{\int_{0}^{t}\normEc{s}ds}\nonumber\\ &&\qquad\qquad+ 
\frac{2\rho}{\delta}(Q+\Gamma_1)\expect{\int_{t}^{t_{\N{t}}}\normEc{s}ds}
+ \hmax^{1-\varepsilon}\GVII.
\label{eq:NtGronwallPreMomentBound}
\end{eqnarray}
Similarly
\begin{eqnarray}
\expect{\|E_{\N{t}-1}\|^2}
&\leq& \frac{2\rho}{\delta}(Q+\Gamma_1)\expect{\int_{0}^{t}\normEc{s}ds}\nonumber \\
& & - \frac{2\rho}{\delta}(Q+\Gamma_1)\expect{\int_{\N{t}-1}^{t}\normEc{s}ds}+\hmax^{1-\varepsilon}\GVII\nonumber\\
&\leq& \frac{2\rho}{\delta}(Q+\Gamma_1)\expect{\int_{0}^{t}\normEc{s}ds} 
+ \hmax^{1-\varepsilon}\GVII.
\label{eq:Nt-1GronwallPreMomentBound}
\end{eqnarray}
By \eqref{eq:interpolant} for all $s\in[t_{\N{t}-1},t_{\N{t}}]$ a.e.,
\begin{equation}\label{eq:e2sumBound}
\normEc{s} \leq \max\{\|E_{\N{t}-1}\|^2,\|E_{\N{t}}\|^2\}\leq \|E_{\N{t}-1}\|^2+\|E_{\N{t}}\|^2.
\end{equation}
Sum both sides of \eqref{eq:NtGronwallPreMomentBound} and \eqref{eq:Nt-1GronwallPreMomentBound} and then use \eqref{eq:e2sumBound} to get
\begin{eqnarray*}
\lefteqn{\expect{\|E_{\N{t}-1}\|^2}+\expect{\|E_{\N{t}}\|^2}}\\
&\leq&\frac{4\rho}{\delta}(Q+\Gamma_1)\expect{\int_{0}^{t}\normEc{s}ds}+\frac{2\rho}{\delta}(Q+\Gamma_1)\expect{\int_{t}^{t_{\N{t}}}\normEc{s}ds}+2\hmax^{1-\varepsilon}\GVII\\
&=&\frac{4\rho}{\delta}(Q+\Gamma_1)\expect{\int_{0}^{t}\normEc{s}ds}\\
&&+\frac{2\rho}{\delta}(Q+\Gamma_1)h_{\max}\left(\expect{\|E_{\N{t}-1}\|^2}+\expect{\|E_{\N{t}}\|^2}\right)+2\hmax^{1-\varepsilon}\GVII\\
\end{eqnarray*}
We can write, by \eqref{eq:deltagamma},
\begin{multline*}
\lefteqn{\expect{\|E_{\N{t}-1}\|^2}+\expect{\|E_{\N{t}}\|^2}}\\
\leq\frac{2\rho(Q+\Gamma_1)/\delta}{1-h_{\max}2\rho(Q+\Gamma_1)/\delta}\expect{\int_{0}^{t}\normEc{s}ds}+2\hmax^{1-\varepsilon}\GVII\frac{1}{1-h_{\max }2\rho(Q+\Gamma_1)/\delta}.
\end{multline*}
Therefore since $t\in[t_{\N{t}-1},t_{\N{t}}]$ and by the lower bound in \eqref{eq:e2sumBound}, 
\begin{equation}\label{eq:noCDGI5}
\expect{\normEc{t}}
\leq\frac{4\rho(Q+\Gamma_1)}{\delta}\gamma\expect{\int_{0}^{t}\normEc{s}ds}+2\hmax^{1-\varepsilon}\gamma\GVII,\quad t\in[0,T],
\end{equation}
where $\gamma$ is given by \eqref{eq:deltagamma} in Assumption \ref{assum:extrahbound}.

Since \eqref{eq:noCDGI5} holds for all $t\in[0,T]$, an application of the integral form of Gronwall's inequality (e.g. \cite[Theorem 8.1]{Mao}) gives, for each $t\in[0,T]$,
\begin{equation*}
\expect{\normEc{t}}\leq \hmax^{1-\varepsilon}\left(2\gamma\GVII\right)
\exp\left(\frac{4\rho(Q+\Gamma_1)}{\delta}\gamma T\right).
\end{equation*}
The statement of the Theorem follows with
$$C_{\varepsilon,m,\rho,T}:=\left(2\gamma\GVII\right)
\exp\left(\frac{4\rho(Q+\Gamma_1)}{\delta}\gamma T\right).$$
}
\end{proof}
\section{A comparative numerical review of some available schemes}\label{sec:num}

Given the semi-linear SDE
\begin{equation}\label{eq:numSDE}
du = \left[Au + f(u)\right] dt + G(u) dW,
\end{equation}
with solution $u:[0,T]\times \Omega \to \noteB{\mathbb{R}^d}$,
we compare our semi-implicit adaptive numerical method to a number of different fixed-step schemes (with time step $h$) that we outline below. Some numerical examples for
an explicit adaptive scheme are given in \cite{KeLo2016}.
The majority of recent developments concentrate on a perturbation of
the flow (or solution) of order $h^{1/2}$
or higher, however the
first method we present is the classic implicit approach.
We do not consider an exhaustive list of taming-type schemes and there
are other variants available, see for example
\cite{HutzenthalerJentzen,Mao15,Mao16,Sabinas2016,SzpruchZhang,ZhangMa}. 
Our examples illustrate
some of the drawbacks of explicit schemes, for example where linear
stability is an issue.

\begin{enumerate}[label={\arabic*.}, wide, labelindent=0pt]
\item{\bf Drift implicit scheme \cite{MaoSzpruch2012}}
This is given for \eqref{eq:numSDE} by 
$$Y_{n+1}=Y_n + \Dt (A Y_{n+1}+f(Y_{n+1})) + g(Y_n) \DW_{n+1}.$$
Although strong convergence has been proved (see \cite{MaoSzpruch2012}), at each step a nonlinear system of the form 
\begin{equation}\label{eq:nonlinSolve}
0=Y_{n+1}- \Dt (A Y_{n+1}+f(Y_{n+1})) + b
\end{equation}
needs to be solved for $Y_{n+1}$  for some vector $b$. Even for
the deterministic case there is no guarantee the nonlinear solver
will converge to the correct root (see \cite[Chapter 4]{StuartHumphries}).
We observe in our numerical experiments that both a 
standard Newton method and the {\sc matlab} nonlinear solver {\tt fsolve}
(or {\tt fzero} in one-dimension) may fail to converge. In the event of a step where this occurs we use
as a backstop an alternative explicit method, in this article taken to
be the balanced method (see below). The drift implicit scheme with this backstop method is denoted by {\tt Drift Implicit} in the figures of this Section. 
Finally, note that for several examples in this section the implicit solver may be made more efficient by exploiting a known closed-form solution for the nonlinear system \eqref{eq:nonlinSolve}. Such solutions are not in general available and so we do not make use of them in our comparative analysis here. 

\item{\bf Tamed \cite{Sabinas2013arxiv}}
A tamed version which may be used when the solutions of \eqref{eq:numSDE} have a limited number of finite moments \cite{ZhangMa}
$$Y_{n+1}=Y_n+ \frac{\Dt AY_n + f(Y_n) + \sum_{j=1}^m g_j(Y_n)\DW} 
{1+ \Dt^{\beta} \|AY_n+f(Y_n)\| + \sum_{j=1}^m \|g_j(Y_n)\|\Dt^{\beta}}.$$
Strong convergence of order $1/2$ is achieved by setting $\beta=1/2$. We denote this method {\tt Tamed}. 

\item{\bf Balanced Method \cite{TretyakovZhang}} is given for \eqref{eq:numSDE} by
$$Y_{n+1}= Y_n + \frac{\Dt (AY_n + f(Y_n)) + \sum_{r=1}^m g_r(Y_n)
  \DW_{r,n+1}}{1+\Dt\|AY_n+f(Y_n)\|+\sum_{r=1}^m \|g_r(Y_n)\DW_{r,n+1}\|}.$$
This was proved to be strongly convergent with order $1/2$ (including
for additive noise) and is denoted in the figures of this Section as {\tt BM}.

\item{\bf Projected EM \cite{beyn16}}  uses the standard EM method when $Y_n$ is inside a ball of radius inversely proportional to the step step size. However, outside of the ball, the numerical solution is projected onto the ball. We have for $Z:=\min(1,Y_n/\sqrt{h}\|Y_n\|)$
$$
Y_{n+1} = Z+h(AZ+f(Z) + \sum_{r=1}^m g_r(Z)\Delta W_{r,n+1}.
$$
This is denoted below as {\tt Projected EM}.

\end{enumerate}

We provide a comparative illustration of the combined effect of semi-implicitness and adaptivity using five
examples ranging from geometric Brownian motion to a system of SDEs arising
from the spatial discretisation of an SPDE. Recall that our use of a
semi-implicit method controls instabilities from a linear 
operator and the adaptive timestepping controls the discretisation
of the nonlinear structure.
Stiffness is manifested in the structure of each of these equations in
different ways: ranging from the linearity only (in geometric Brownian
motion) to both in the linear operator and nonlinearities for a
discretisation of an SPDE.

To examine strong convergence in $\hmax$ for the SDE examples below we solve with $M=1000$ samples to estimate the root mean square error (RMSE) at a final time $T=Nh=1$, $\sqrt{\expect{\|X(T)-X_N\|^2}}$ and we estimate the standard deviation from $20$ groups of $50$ samples included on the error plots. Reference solutions are computed with $10^6$ uniform steps on $[0,T]$.
For efficiency we compare the RMSE against the average computing time over the $1000$ samples (denoted cputime). Unless otherwise stated we take $\rho=10$ throughout.

\subsection{Geometric Brownian Motion}\label{Sec:GBM}
The classic example to illustrate linear mean square stability is
geometric Brownian motion 
\begin{equation}\label{eq:GBM}
du(t) = ru(t) dt + \sigma u(t) dW(t), \quad u(0)=u_0,\quad t\geq 0.
\end{equation}
If $r+\sigma^2/2<0$ it is straightforward to see that
$\expect{(u(t)^2}\to 0$ as $t\to \infty$ and that the (fixed step)
explicit Euler method is only stable if $0<\Dt<-2(r+\sigma^2/2)/r^2$.
The drift and diffusion are both linear functions, so there is
no need for either taming or adaptivity to control growth from a nonlinear term; indeed  
in this example the semi-implicit adaptive and fully drift implicit schemes co-coincide if
$A=r$ and $f(u)=0$.

\begin{figure}
  \begin{center}
    (a) \hspace{0.48\textwidth} (b)\\
    \includegraphics[width=0.48\textwidth,height=0.25\textheight]{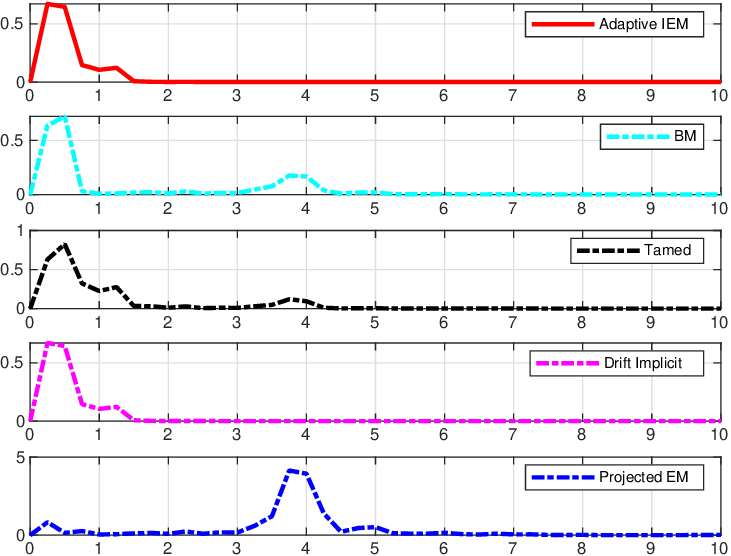}
    \includegraphics[width=0.48\textwidth,height=0.25\textheight]{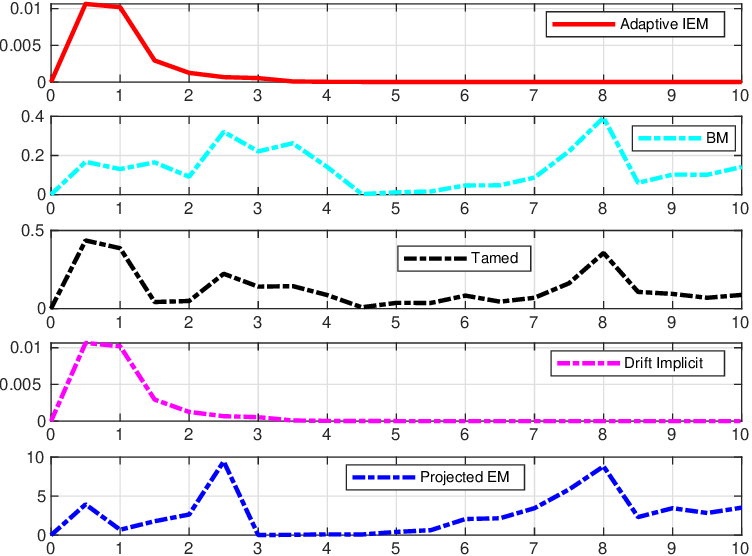}
    \caption{Error $|u(t_n)-Y_n|^2$ plotted against $t_n\in[0,10]$ in two sample paths for geometric Brownian motion \eqref{eq:GBM} with (a) $\hmax=0.25$ and (b)
      $\hmax=0.5$.}\label{fig:GBM}
  \end{center}
\end{figure}

However it is instructive to compare the explicit schemes to the implicit schemes ({\tt Adaptive IEM} and {\tt Drift Implicit}). We take $r=-8$ and $\sigma=3$ so that the explicit
Euler method is unstable for $\Dt=0.25$ and $\Dt=0.5$.
In Fig. \ref{fig:GBM} we plot 
the error squared, $|u(t_n)-Y_n|^2$, of two sample paths one with $\hmax=0.25$ (a) and
$\hmax=0.5$ (b).  Although the tamed and projected schemes control growth from the
linear \noteB{instability, we observe that this control} can come at a price of bounded oscillations.

\subsection{1D stochastic Ginzburg-Landau}\label{sec:GL}
The 1D stochastic Ginzburg-Landau SDE is a classic example with a
cubic nonlinearity in the drift and linear diffusion
\begin{equation}\label{eq:SGL}
dX(t)=aX(t)[b-X(t)^2]dt + cX(t)dW(t),\quad t\geq 0.
\end{equation}

\begin{figure}
  \begin{center}
    (a) \hspace{0.48\textwidth} (b)\\
       \includegraphics[width=0.48\textwidth,height=0.25\textheight]{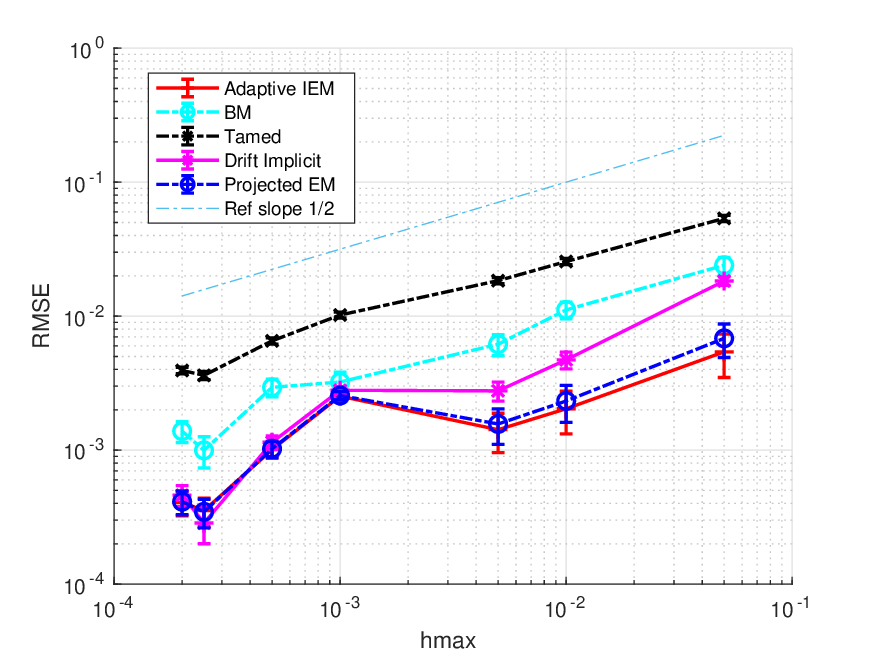}
    \includegraphics[width=0.48\textwidth,height=0.25\textheight]{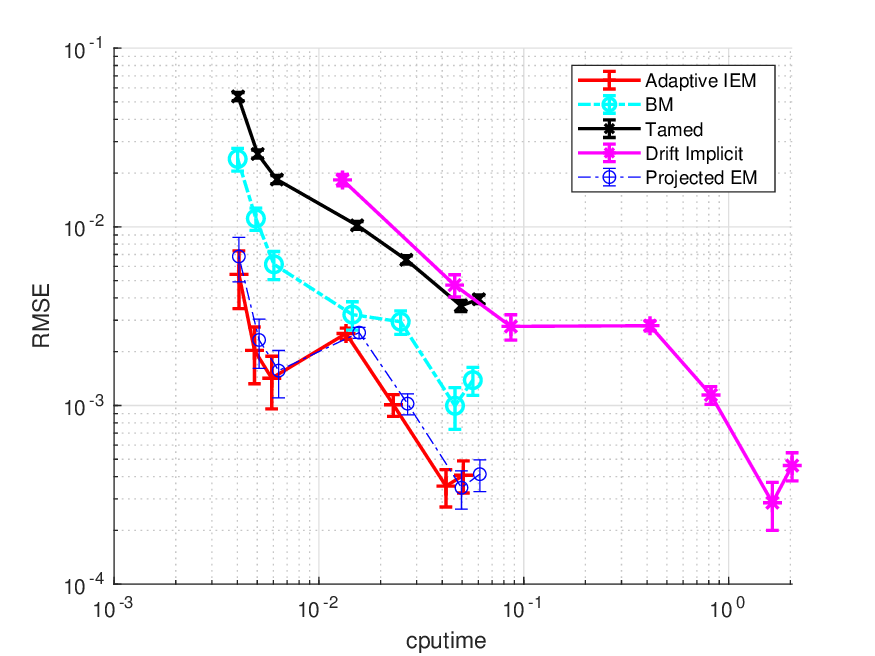}
    \caption{Convergence and efficiency of methods applied to the
      stochastic Ginzburg Landau equation \eqref{eq:SGL}. We compare
      RMSE at $T=1$ against $\hmax$ in (a) and efficiency (RMSE vs cputime) in (b). 
      }
      \label{fig:Cubic1}
  \end{center}
  
\end{figure}

We take here parameter values as in \cite[Example 4.7]{Mao15}, $a=0.1$, $b=1$ and $c=0.2$, $x(0)=2$, and solve to $T=1$. 
We see in Fig. \ref{fig:Cubic1} that all the methods demonstrate
convergence and that {\tt Adaptive IEM} and {\tt Projected EM}
are similar in convergence and efficiency. Neither the adaptive nor drift-implicit schemes used the
backstop method. 

\subsection{Stochastic volatility system.}\label{sec:SVS}
We consider an extension of the $3/2$-volatility model to
two dimensions as in \cite{Sabinas2016}
\begin{equation}\label{eq:SVol}
dX(t)=\lambda X(t)[\mu-|X(t)|]dt + \Sigma|X(t)|^{3/2} dW(t),\quad t\geq 0,
\end{equation}
with $\lambda=2.5$, $\mu=1$, $X(0)=a[2,2]^T$, $a=1,10,100$, and  
$$\Sigma=\left( \begin{array}{cc} \frac{2}{\sqrt{10}} &  \frac{1}{\sqrt{10}} \\
\frac{1}{\sqrt{10}} & \frac{2}{\sqrt{10}} 
\end{array}\right).$$

\begin{figure}
  \begin{center}
    (a) \hspace{0.48\textwidth} (b)\\
     \includegraphics[width=0.48\textwidth,height=0.25\textheight]{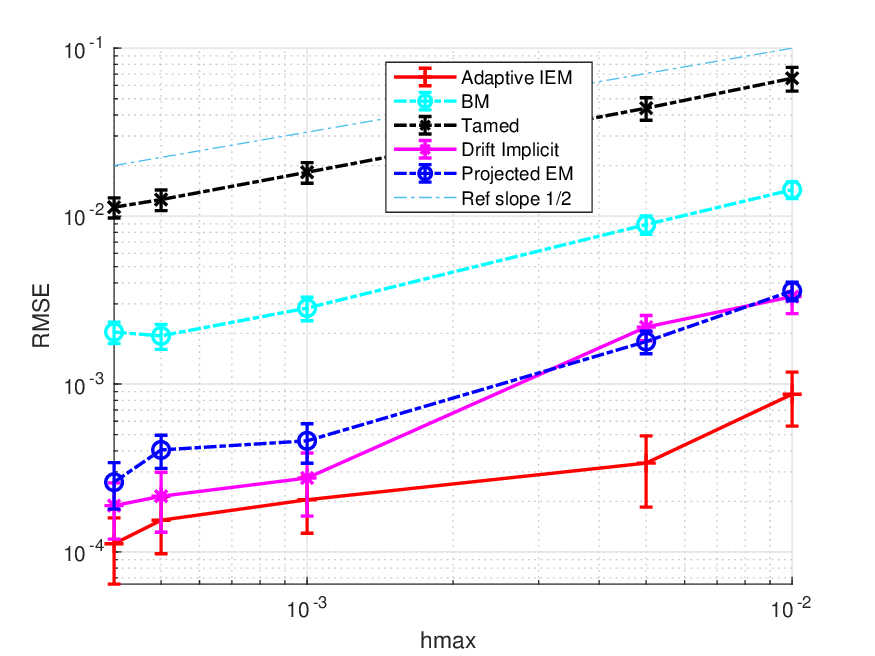}
        \includegraphics[width=0.48\textwidth,height=0.25\textheight]{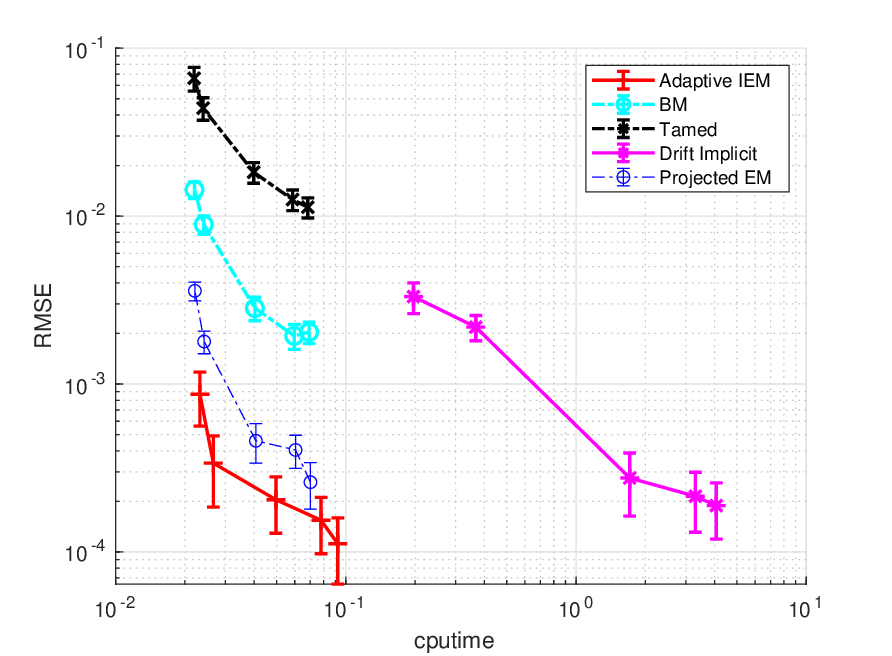}
    \caption{Convergence of methods applied to the stochastic volatility system \eqref{eq:SVol}. We compare RMSE at $T=1$ against $\hmax$ in (a) and efficiency (RMSE vs cputime) in (b).}
    \label{fig:StochVol2}
  \end{center}
  
\end{figure}

\begin{figure}
  \begin{center}
    (a) \hspace{0.48\textwidth} (b)\\
     \includegraphics[width=0.48\textwidth,height=0.25\textheight]{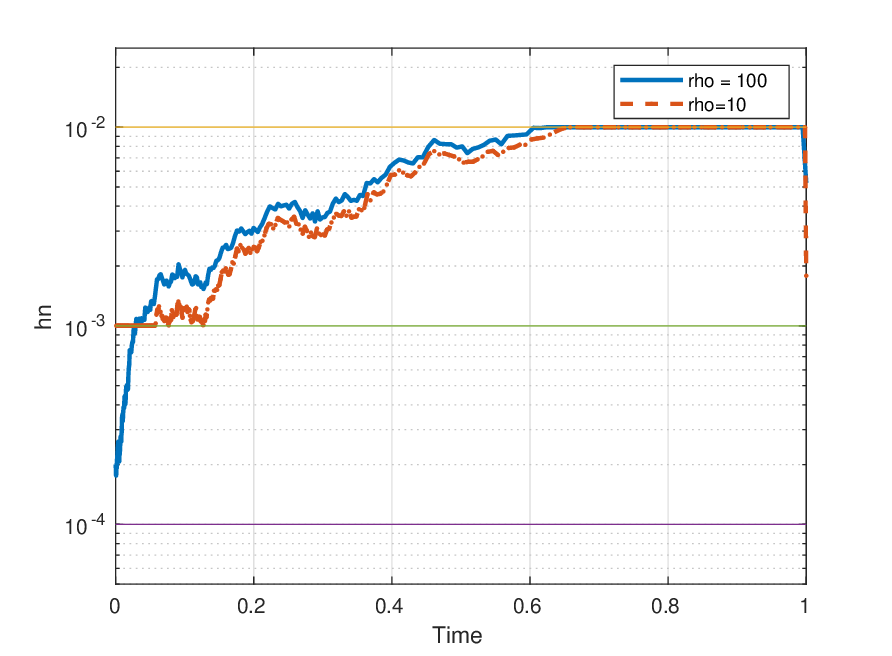}
        \includegraphics[width=0.48\textwidth,height=0.25\textheight]{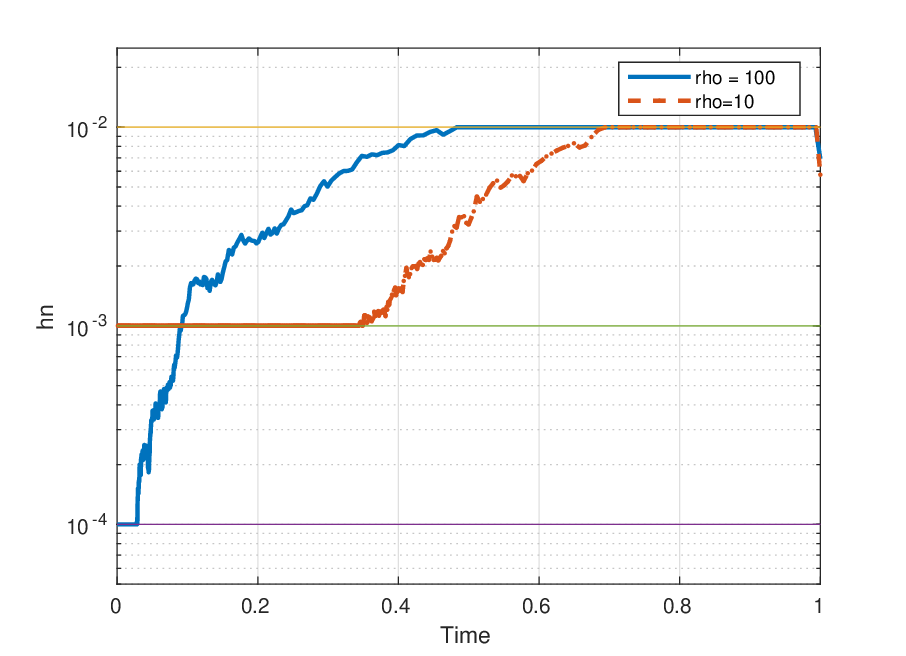}
    \caption{Selected time steps $h_n$ with $\hmax=0.01$ and for $\rho=10,100$. In (a) $X(0)=[20,20]^T$ and in (b) $X(0)=[200,200]^T$.}\label{fig:SVDt}
  \end{center}
  
\end{figure}

We see in Fig. \ref{fig:StochVol2} that 
all the methods demonstrate convergence but that {\tt BM} and {\tt Tamed} have a larger error constant and 
the adaptive method {\tt \noteB{Adaptive} IEM} is the most efficient. The initial data taken was $X(0)=[2,2]^T$ and 
the  backstop method was not used for either drift implicit or adaptive methods (as for \eqref{eq:SGL}). 
In Fig. \ref{fig:SVDt} we examine the time steps $h_n$ for a single noise path with the same value of $\hmax=0.01$ but with $\rho=10$ and $\rho=100$ corresponding to $\hmin=10^{-3}$ and $\hmin=10^{-4}$. In (a) we have initial data $X(0)=[20,20]^T$ and in (b) $X(0)=[200,200]^T$. In (a) we observe that with $\rho=100$ the backstop method is not used at all (but it is required for $\rho=10$) whereas in (b), to deal with the larger initial data, the backstop is used in both cases. As time progresses the time-step taken increases until it reaches $\hmax$. 

These observations suggest that practitioners who apply a standard explicit or semi-implicit Euler-Maruyama scheme over a uniform mesh with a step size sufficiently
small (e.g. close to $h_{\text{min}}$ with large $\rho$) may rarely encounter the spurious coefficient responses that underlie the lack of strong convergence for the scheme.

\subsection{Finite difference approximation of an SPDE}\label{sec:SPDE}
Consider the SPDE
\begin{equation}\label{eq:SPDE}
du = \left[ \epsilon u_{xx} + \eta u + u^3 - \lambda u^5\right] dt +
\sigma u^2d\widehat{W}
\end{equation}
with $t\geq 0$, $x\in[0,1]$ and zero Dirichlet boundary conditions.
We take initial data $u_0(x)=2\sin(\pi x)$, $\sigma=0.2$, $\eta=11$,
$\lambda=2$ and trace class noise $\widehat W$
$$\widehat W(x,t)=\sum_{j=1}^{m} j^{-3/2}\sin(j\pi x) W_j(t),\quad t\geq 0,\quad x\in[0,1],$$
for some $m\in \mathbb{N}\setminus\{0\}$ and where $W_j(t)$ are mutually independent standard Brownian motions.

We introduce a grid of $d+2$ uniformly spaced points $x_k=k\Dx$ on $[0,1]$ for $k=1,\ldots,d+2$.
Then the finite difference approximation in space leads to a system of $d$ SDEs:
\begin{equation}\label{eq:FDASPDE}
d\vecu(t) = \left[ \epsilon A\vecu(t) + \eta \vecu(t) + \vecu(t)^3 - \lambda \vecu(t)^5 \right] dt +
\sigma \vecu(t)^2 d\vec{W}(t),\quad t\geq 0,
\end{equation}
where we denote $\vecu:=(u_1,u_2,\ldots,u_{d})^T$, $u_k(t)\approx u(x_k,t)$ and the noise process is
$$\vec{W}:=(\widehat{W}(x_1,t),\widehat{W}(x_2,t),\ldots,\widehat{W}(x_{d},t))^T.$$ 
The tri-diagonal matrix $A$ is the standard finite difference approximation to the Laplacian with zero Dirichlet boundary conditions given by
$$
A = \frac{1}{\Dx^2}\left( 
\begin{array}{ccccc}2 & -1 &  &  & \\ 
-1 & 2 & -1 & \\
& \ddots & \ddots & \ddots & \\
&        &        &    -1  & 2

\end{array} \right) \in \mathbb{R}^{ d\times d}.
$$
For further details on the finite difference approximation of SPDEs see \cite{LPS}.
To be able to compare different system sizes $d$ we scaled the Euclidean norm in the standard way (see for example \cite{LPS}) to approximate the function space norm $L^2([0,1])$.
This system of SDEs displays linear stiffness (similar to the
geometric Brownian motion) and nonlinear stiffness arising from the drift and diffusion coefficients.
The parameter $\epsilon$ then determines the degree of linear
stiffness. 
To examine convergence to the SDE system \eqref{eq:FDASPDE}  we take $\epsilon=0.1$, $T=1$ for  $d=m=10$ (Fig. \ref{fig:SPDE10}) and $d=m=100$ (Fig. \ref{fig:SPDE100}). 
\begin{figure}
  \begin{center}
    (a) \hspace{0.48\textwidth} (b)\\
    \includegraphics[width=0.48\textwidth,height=0.25\textheight]{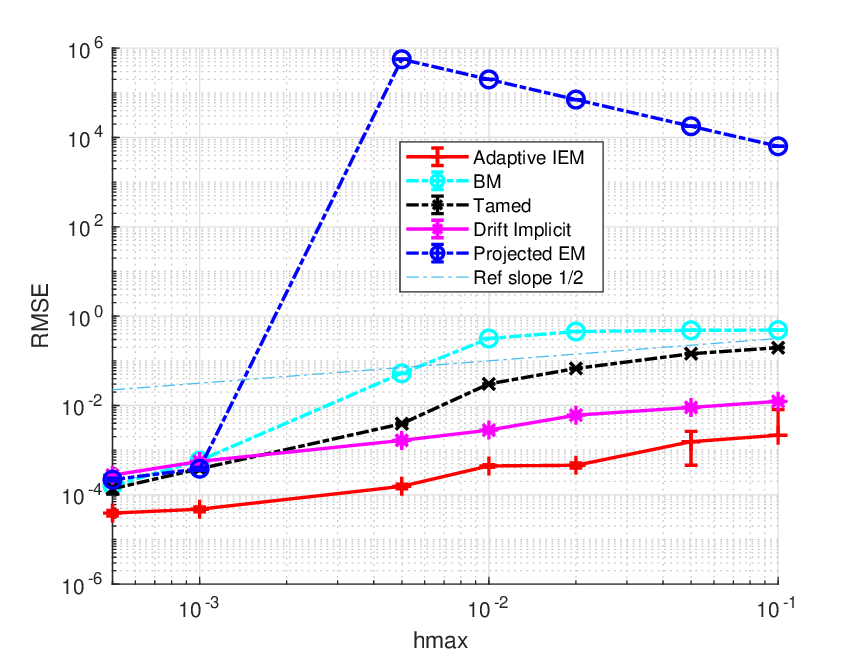}
    \includegraphics[width=0.48\textwidth,height=0.25\textheight]{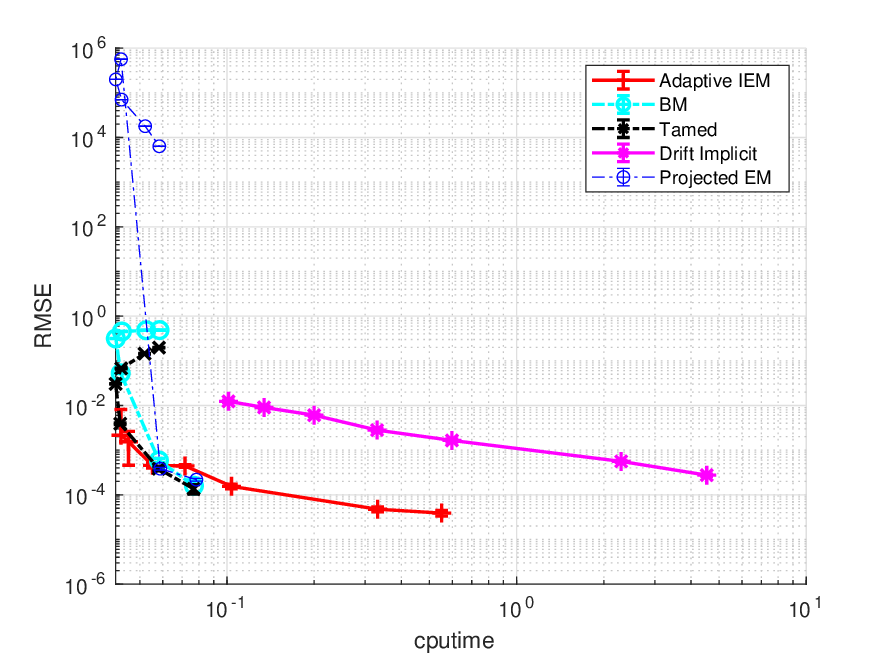}
    \caption{Convergence and efficiency for the methods applied to the
      finite difference  approximation of the SPDE given by
      \eqref{eq:FDASPDE} with $d=10$ (a) RMSE vs $\hmax$ (with reference line
      of slope 0.5) and (b) the efficiency (RMSE vs cputime).}\label{fig:SPDE10}
  \end{center}
\end{figure}

For the smaller system size ($d=10$) in Fig. \ref{fig:SPDE10} (a) we see all methods converging for $\hmax$ sufficiently small, although  {\tt Projected EM} initially diverges. We see in (b) that {\tt Adaptive IEM} is more efficient than {\tt Drift Implicit} and is similar to the other explicit schemes. When $m=100$ {\tt Projected EM} is no longer seen to converge on this range of $h$ and so is not plotted in Fig. \ref{fig:SPDE100}. In (a) we now see that the {\tt Drift Implicit} is more accurate \noteB{for a given $h_{\max}$} and from (b) that it is comparable in efficiency with {\tt Adaptive IEM}.
\begin{figure}
  \begin{center}
        (a) \hspace{0.48\textwidth} (b)\\   
    \includegraphics[width=0.48\textwidth,height=0.25\textheight]{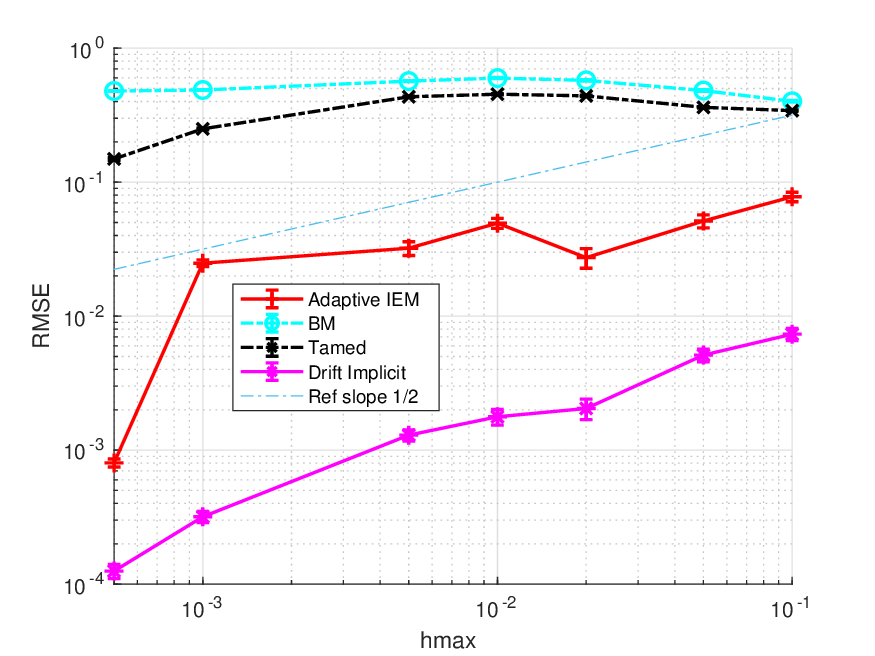}
    \includegraphics[width=0.48\textwidth,height=0.25\textheight]{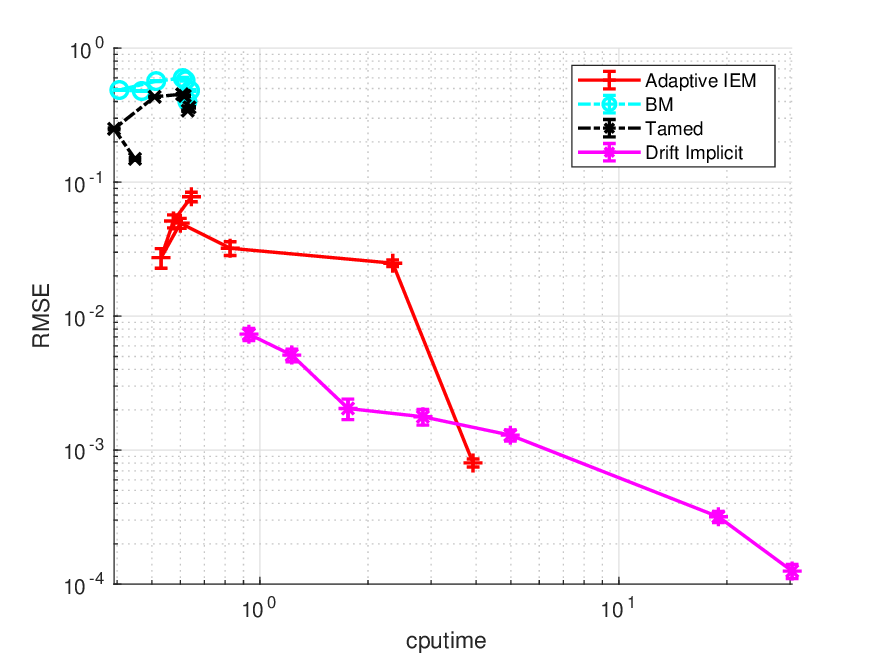}
    \caption{Convergence and efficiency for the methods applied to the
      finite difference approximation of the SPDE given by
      \eqref{eq:FDASPDE} with $d=100$ (a) RMSE vs $\hmax$ (with reference line
      of slope 0.5) and (b) the efficiency (RMSE vs cputime).}\label{fig:SPDE100}
  \end{center}
  
\end{figure}

\section{Proofs of Technical Results}\label{app:technical}
In this section we frequently use the inequality $\|a+b\|^p\leq 2^{p}(\|a\|^p+\|b\|^p)$, where $a,b\in\mathbb{R}^d$, and $p\in\mathbb{R}^+$, which follows from $
\|a+b\|^p\leq (\|a\|+\|b\|)^p\leq (2(\|a\|\vee \|b\|))^p=2^p(\|a\|^p\vee \|b\|^p)\leq 2^p(\|a\|^p+\|b\|^p)$.

\begin{proof}[Lemma \ref{lem:Xregularity}]
Fix $n\in\mathbb{N}$ and suppose that \noteB{$t_n<s\leq t_{n+1}$}. Then
$$X(s)-X(t_n)=\int_{t_n}^s [AX(r)+f(X(r))]dr+\int_{t_n}^s g(X(r))dW(r).$$
By the triangle inequality and ~\cite[Theorem 1.7.1]{Mao} (with conditioning on $\mathcal{F}_{t_n}$),
\begin{eqnarray*}
\lefteqn{\expect{\|X(s)-X(t_n)\|^p | \mathcal{F}_{t_n}}}\\ &\leq& 2^p \expect{\left\| \int_{t_n}^s [AX(r)+f(X(r))] dr\right\|^p \bigg| \mathcal{F}_{t_n}}
+ 2^p \expect{\left\|\int_{t_n}^s  g(X(r)) dW(r)\right\|^p \bigg| \mathcal{F}_{t_n}}\\
&\leq& 2^p  |s-t_n|^{p-1}\int_{t_n}^s \expect{\left\|AX(r)+f(X(r))\right\|^p | \mathcal{F}_{t_n}}dr \\
&&\qquad\qquad+ 2^{p/2}p^{p/2}(p-1)^{p/2} |s-t_n|^{p/2-1}\int_{t_n}^s \expect{\| g(X(r))\|^p_F | \mathcal{F}_{t_n}} dr,\quad a.s.
\end{eqnarray*}
Next, we apply \eqref{eq:fgBound}, Lemma \ref{lem:condSDEMoments}, and the fact that $\noteB{\|A\|_F^p}<\infty$, to get a.s.
\begin{eqnarray*}
\lefteqn{\expect{\|X(s)-X(t_n)\|^p | \mathcal{F}_{t_n}}}\\ 
&\leq& 2^{2p} |s-t_n|^{p-1} \int_{t_n}^s \expect{\noteB{\|A\|_F^p}\|X(r)\|^p+c_3^2(1+\|X(r)\|^{p\gamma_0+p}) | \mathcal{F}_{t_n}}dr\\ 
&&+2^{p/2}p^{p/2}(p-1)^{p/2}|s-t_n|^{p/2-1} \int_{t_n}^s\expect{ {c_4}^p(1+\|X(r)\|^{p\gamma_1+p})| \mathcal{F}_{t_n}}dr\\
&\leq&2^{2p}|s-t_n|^{p-1}\int_{t_n}^{s}\left[\noteB{\|A\|_F^p}(\nu_1+\nu_2\|X(t_n)\|^p)\right]dr\\
&&+2^{2p}|s-t_n|^{p-1}\int_{t_n}^{s}\left[c_3^2(1+(\nu_1+\nu_2\|X(t_n)\|^{p\gamma_0+p}))\right]dr\\
&&+2^{p/2}p^{p/2}(p-1)^{p/2}|s-t_n|^{p/2-1}\int_{t_n}^{s}\left(c_4^p(1+(\nu_1+\nu_2\|X(t_n)\|^{p\gamma_1+p}))\right)dr.
\end{eqnarray*}

Therefore, since $|s-t_n|\leq 1$, we can define an a.s. finite and $\mathcal{F}_{t_n}$-measurable random variable 
\begin{multline}\label{eq:Ln}
\bar{L}_{n}:=2^{2p}\noteB{\|A\|_F^p}(\nu_1+\nu_2\|X(t_n)\|^p+c_3^p(1+(\nu_1+\nu_2\|X(t_n)\|^{p\gamma_0+p})))\\+2^{p/2}p^{p/2}(p-1)^{p/2}c_4^2(1+(\nu_1+\nu_2\|X(t_n)\|^{p\gamma_1+p})),
\end{multline}
so that \eqref{eq:holdBound} holds.
\end{proof}

\begin{proof}[Lemma \ref{LPSbook}]
Part (i): Set $\gamma:=\gamma_0\vee \gamma_1$, where $\gamma_0,\gamma_1$ are as in Assumptions \ref{assum:polybounds} and \ref{assum:p}, and for $q=1,2,\ldots$ define
\[
A_{q}(s,t_n):=(1+2^{2\gamma}\|X(t_n)\|^{2\gamma}+2^{2\gamma}\|X(s)-X(t_n)\|^{2\gamma})^{2^{q-1}}\|X(s)-X(t_n)\|,
\]
which satisfies the relation
\[
A_{q}(s,t_n)^2=A_{q+1}(s,t_n)\|X(s)-X(t_n)\|, \quad q\in\mathbb{N},
\]
and, by Lemma \ref{lem:condSDEMoments}, the a.s. finite bound,
\begin{multline}\label{eq:AqBound}
\expect{A_{q}(s,t_n)\|X(s)-X(t_n)\||\mathcal{F}_{t_n}}\\
\leq \left(2^{(2^{q-1}+2)}(1+2^{(6\gamma+2^q)}\nu_1)\|X(t_n)\|^4+2^{(2\gamma+2^{q-1}+3)}\|X(t_n)\|^{4\gamma+2^q+4}\right.\\
\left.+2^{(6\gamma+2^q+2^{q-1}+2)}\|X(t_n)\|^{4\gamma+2^q+4}+2^{(6\gamma+2^q+2^{q-1}+2)}\nu_2\|X(t_n)\|^{4\gamma+2^q+4}\right)^{1/2},
\end{multline}
the right-hand-side of which we denote $(\bar \Upsilon_{q,n})^{1/2}$, where $\bar\Upsilon_{q,n}$ is an a.s. finite and $\mathcal{F}_{t_n}$-measurable random variable. Let $q\in\mathbb{N}\setminus\{0\}$ satisfy  $q>1-\log_2\varepsilon$. Then by \eqref{eq:fgDer} and $q$ successive applications of the Cauchy-Schwarz inequality, and \eqref{eq:holdBound} with $p=2$ in the statement of Lemma \ref{lem:Xregularity}, we get
\begin{eqnarray*}
\lefteqn{\mathbb{E}\left[\left\|\int_{t_n}^{t_{n+1}}R_f(s,t_n,X(t_n))ds\right\|^2 \bigg| \mathcal{F}_{t_n}\right]}\\
&\leq&h_{n}\int_{t_n}^{t_{n+1}}\expect{\|R_f(s,t_n,X(t_n))\|^2|\mathcal{F}_{t_n}}ds\\
&\leq& 2c_1^2h_{n}\int_{t_n}^{t_{n+1}}\expect{A_1(s,t_n)\|X(s)-X(t_n)\||\mathcal{F}_{t_n}}ds\\
&\leq& 2c_1^2h_{n}\int_{t_n}^{t_{n+1}}\left(\mathbb{E}[A_q(s,t_n)\|X(s)-X(t_n)\||\mathcal{F}_{t_n}]\right)^{1/(2^{q-1})}\\
&&\qquad\qquad\qquad\qquad\qquad\qquad\qquad\times(\bar L_{2,n}|s-t_n|)^{\sum_{i=2}^{q}1/(2^{i-1})}ds\\
&\leq&2c_1^2\bar \Upsilon_{q,n}^{1/(2^q)}\bar L_{2,n}^{\sum_{i=2}^{q}1/(2^{i-1})}h_{n}^{2+\sum_{i=2}^{q}1/(2^{i-1})}\\
&\leq& \bar\Lambda_{\varepsilon,n} h_{n}^{3-\varepsilon},\quad a.s.,
\end{eqnarray*}
where  $\bar\Lambda_{\varepsilon,n}:=2c_1^2\bar \Upsilon_{q,n}^{1/(2^q)}\bar L_{2,n}^{\sum_{i=2}^{q}1/(2^{i-1})}$ depends on $\varepsilon$ through $q$.

Part (ii): By the conditional form of the It\^o isometry, for $i=1,\ldots,m$,
\begin{equation*}
\expect{\left\|\int_{t_n}^{t_{n+1}}R_{g_i}(s,t_n,X(t_n))dW(s)\right\|^2\bigg|\mathcal{F}_{t_n}}=\int_{t_n}^{t_{n+1}}\expect{\|R_{g_i}(s,t_n,X(t_n))\|^2|\mathcal{F}_{t_n}}ds,
\end{equation*}
and the proof follows as in Part (i), with a reduction of one in the order of $h_{n}$. 

\noteB{Part (iii) holds as a special case of Part (i).}
Part (iv) follows by an application of the Cauchy-Schwarz inequality, followed by Jensen's inequality for the functions $(\cdot)^{1/(2^{q-1})}$ and $(\cdot)^{\sum_{i=2}^{q}1/(2^{i-1})}$ (both of which are concave over $\mathbb{R}^+$, by the second derivative test), to get
\begin{eqnarray*}
\expect{\bar\Lambda_{\varepsilon,n}}&=&2c_1^2\expect{\bar \Upsilon_{q,n}^{1/(2^q)}\bar{L}_{2,n}^{\sum_{i=2}^{q}1/(2^{i-1})}}\\
&\leq& 2c_1^2\sqrt{\expect{\bar \Upsilon_{q,n}^{1/(2^{q-1})}}}\sqrt{\expect{\left(\bar{L}_{2,n}^{\sum_{i=2}^{q}1/(2^{i-1})}\right)^{2}}}\\
&\leq& 2c_1^2\left(\expect{\bar \Upsilon_{q,n}}\right)^{1/(2^q)}\sqrt{\expect{(\bar{L}_{2,n}^2)^{\sum_{i=2}^{q}1/(2^{i-1})}}}\\
&\leq&2c_1^2\left(\expect{\bar \Upsilon_{q,n}}\right)^{1/(2^q)}\left(\expect{\bar{L}_{2,n}^2}\right)^{\sum_{i=2}^{q}1/(2^{i})},
\end{eqnarray*}
\noteB{which is finite under the conditions of Assumption \ref{assum:p}: $p$ satisfies \eqref{eq:pLowerBound}, and therefore by \eqref{eq:momStrong} the finiteness of $\expect{\bar \Upsilon_{q,n}}$ is ensured by \eqref{eq:AqBound} and that of $\expect{\bar L^2_n}$ is ensured by \eqref{eq:Ln}. }
\end{proof}

\begin{remark}\label{rem:noStrongMom}
By making $q$ successive applications of the Cauchy-Schwarz inequality in the proof of Lemma \ref{LPSbook} we separate the expectation of dependent random factors in $R_f$ and $R_g$ in such a way that the highest possible order of $h_{n}$ is achieved in the estimate, given the available finite moment bounds. This is necessary to ensure a polynomial order of strong convergence in the statement of Theorem \ref{thm:adaptConv}. If the diffusion coefficient $g$ is globally Lipschitz continuous then the resulting uniform bound on each $\|Dg_i(x)\|_F$, along with stronger moment bounds of the form \eqref{eq:momStrong}, sidesteps that requirement. In this case the statement of Lemma \ref{LPSbook}, and hence the statement of Theorem \ref{thm:adaptConv}, would hold with $\varepsilon=0$ (and order constant independent of $q$, and therefore $\varepsilon$), giving an order of strong convergence of $1/2$ for the semi-implicit method with backstop \eqref{eq:finalScheme}, using an admissible timestepping strategy. If we then set $A=0$ in \eqref{eq:SDE}, our method becomes explicit and we recover the main result of \cite{KeLo2016}.
\end{remark}

\section{Conclusion}\label{sec:concl}
The discretisation of SDEs with non-Lipschitz drift and diffusion coefficients is a challenging numerical
problem. We have proved strong convergence for both adaptive semi-implicit and
explicit Euler schemes, and presented numerical results that indicate
the semi-implicit variant is well suited as a general purpose solver, being
more robust than several competing explicit fixed-step methods 
and more efficient than the drift implicit method. 

Both the drift implicit and the adaptive scheme make use of a
backstop method which is triggered when the adaptive timestepping strategy attempts to select a timestep at the minimum stepsize $h_{\text{min}}$. Our numerical experiments indicate that, for an appropriate choice of $\rho$, $h_{\text{min}}$
may be achieved only rarely (if at all). It may be possible to characterise the probability of this occurrence and, if it can be bounded appropriately, a strong convergence result may
be possible for a numerical method of the form \eqref{eq:Scheme} that
does not rely on a backstop method (provided $T$ is reached in a
finite number of steps). A step in this direction may be found in \cite{KeLoSu2019}.

SDEs where the drift coefficient is both positive and non-globally Lipschitz continuous are not covered by the analysis in this article, though adaptive meshes have been used to reproduce 
positivity of solutions with high probability and a.s. stability and instability of equilibria in \cite{KRR2017} (informed by the approach of Liu \& Mao~\cite{LM}). We are unaware of any strong convergence results for such equations.

Finally, since our analysis relies upon the boundedness of \noteB{$\|A\|_F$},
and since the error constant in the strong convergence estimate
increases without bound with the number of independent noise
terms $m$, the results of the article do not automatically extend to
SPDEs. This setting is now considered in \cite{LoStu2019}.

\section*{Acknowledgements}
Gabriel Lord was partially funded by EPSRC grant EP/K034154/1.
The authors are grateful to Stuart Campbell of Heriot-Watt University,  and Alexandra Rodkina of the University of the West Indies at Mona, for useful discussions in preparing this work. \noteB{We also wish to acknowledge the careful reading and detailed feedback provided by anonymous referees, which significantly contributed to the present form of the manuscript.}
\bibliographystyle{siamplain}

\end{document}